\documentclass[11pt, reqno]{amsart}

\usepackage{amsmath, amsfonts, amssymb, amsbsy, bigstrut, graphicx, enumerate,  upref}

\usepackage[numbers, sort&compress]{natbib} 

\usepackage{pstricks}

\newcommand{\ep}{\epsilon}

\newtheorem{thm}{Theorem}[section]

\newtheorem{prop}[thm]{Proposition}
\theoremstyle{definition}

\newtheorem{ex}[thm]{Example}

\newcommand{\cov}{\mathrm{Cov}}

\newcommand{\ee}{\mathbb{E}}

\newcommand{\mx}{\mathcal{X}}

\newcommand{\pp}{\mathbb{P}}

\newcommand{\ra}{\rightarrow}
\newcommand{\rr}{\mathbb{R}}

\newcommand{\tr}{\operatorname{Tr}}

\newcommand{\var}{\mathrm{Var}}

\numberwithin{equation}{section}

\usepackage{hyperref}

\begin{document}
\title{The sample size required in importance sampling}
\author{Sourav Chatterjee}
\address{\newline Department of Statistics \newline Stanford University\newline Sequoia Hall, 390 Serra Mall \newline Stanford, CA 94305\newline \newline \textup{\tt souravc@stanford.edu}\newline \textup{\tt diaconis@math.stanford.edu}}
\thanks{Sourav Chatterjee's research was partially supported by NSF grant DMS-1441513}
\thanks{Persi Diaconis's research was partially supported by NSF grant DMS-1208775}
\author{Persi Diaconis}

\begin{abstract}
The goal of importance sampling is to estimate the expected value of a given function with respect to a probability measure $\nu$ using a random sample of size $n$ drawn from a different probability measure $\mu$. If the two measures $\mu$ and $\nu$ are nearly singular with respect to each other, which is often the case in practice, the sample size required for accurate estimation is large. In this article it is shown that in a  fairly general setting, a sample of size approximately $\exp(D(\nu||\mu))$ is necessary and sufficient for accurate estimation by importance sampling, where $D(\nu||\mu)$ is the Kullback--Leibler divergence of $\mu$ from $\nu$. In particular, the required sample size exhibits a kind of cut-off in the logarithmic scale. The theory is applied to obtain a general formula for the sample size required in importance sampling for one-parameter exponential families (Gibbs measures). 
\end{abstract}

\maketitle


\section{Theory}\label{theory}
Let $\mu$ and $\nu$ be two probability measures on a set $\mx$ equipped with some sigma-algebra. Suppose that $\nu$ is absolutely continuous with respect to $\mu$. Let $\rho$ be the probability density of $\nu$ with respect to $\mu$. Let $X_1,X_2,\ldots$ be a sequence of $\mx$-valued random variables with law $\mu$. Let $f:\mx \ra \rr$ be a measurable function. Suppose that our goal is to evaluate the integral 
\[
I(f) := \int_{\mx} f(y) d\nu(y)\,.
\]
The {\it importance sampling estimate} of this quantity based on the sample $X_1,\ldots, X_n$ is given by 
\[
I_n(f) := \frac{1}{n}\sum_{i=1}^n f(X_i) \rho(X_i)\,.
\]
Sometimes, when the probability density $\rho$ is known only up to a normalizing constant --- that is, $\rho(x)= C \tau(x)$ where $\tau$ is explicit but $C$ is hard to calculate --- the following alternative estimate is used:
\begin{equation}\label{jndef}
J_n(f) := \frac{\sum_{i=1}^n f(X_i) \tau(X_i)}{\sum_{i=1}^n \tau(X_i)}\,.
\end{equation}
It is easy to see that
\[
\ee(I_n(f)) = \int_{\mx} f(x) \rho(x) d\mu(x) = \int_{\mx} f(y)d\nu(y)\,.
\]
Therefore, the expected value of $I_n(f)$ is the quantity $I(f)$ that we are trying to estimate. However, $I_n(f)$ may have large fluctuations. The two main problems in importance sampling are: (a) given $\mu$, $\nu$ and $f$, to determine the sample size required for getting a reliable estimate, and (b) given $\nu$ and $f$, to find a sampling measure $\mu$ that minimizes the required sample size among a given class of measures. We address the first problem in this paper. 

A straightforward approach for computing an upper bound on the required sample size is to compute the variance of $I_n(f)$. Indeed, this is easy to compute:
\begin{align}
\var(I_n(f)) &= \frac{1}{n}\biggl(\int_{\mx} f(x)^2 \rho(x)^2 d\mu(x) - I(f)^2\biggr)\nonumber\\
&=  \frac{1}{n}\biggl(\int_{\mx} f(y)^2 \rho(y) d\nu(y) - I(f)^2\biggr)\,.\label{varform}
\end{align}
The formula for the variance can be used, at least in theory, to calculate a sample size that is sufficient for guaranteeing any desired degree of accuracy for the importance sampling estimate. In practice, however, this number is often much larger than what is actually required for good performance. 

Sometimes the variance formula \eqref{varform} is estimated using the simulated data $X_1,\ldots, X_n$. This estimate is known as the empirical variance. There is an inherent unreliability in using the empirical variance to determine convergence of importance sampling. We will elaborate on this in Section \ref{problem}.

We begin by stating our main theorems. Proofs are collected together in Section \ref{proofs}. A literature review on importance sampling is given at the end of this introduction.

There are three main results in this article. The first  theorem, stated below,  says that under a certain condition that often holds in practice, the sample size $n$ required for $|I_n(f)-I(f)|$ to be close to zero with high probability is roughly $\exp(D(\nu||\mu))$ where $D(\nu||\mu)$ is the Kullback--Leibler divergence of $\mu$ from $\nu$. More precisely, it says that if $s$ is the typical order of fluctuations of $\log \rho(Y)$ around its expected value, then a sample of size $\exp(D(\nu||\mu) + O(s))$ is sufficient and a sample of size $\exp(D(\nu||\mu)-O(s))$ is  necessary for $|I_n(f)-I(f)|$  to be close to zero with high probability. The necessity is proved by considering the worst possible $f$ --- which, as it turns out, is the function that is identically equal to $1$. 

An immediate concern that the reader may have is that $|I_n(f)-I(f)|\approx 0$ may not always be the desired criterion for convergence. If $I(f)$ is very small, then one may want to have $I_n(f)/I(f)\approx 1$ instead. A necessary and sufficient condition for this, when $f$ is the indicator of a rare event, is given in Theorem \ref{rarethm} later in this section. 
\begin{thm}\label{impthm}
Let $\mx$, $\mu$, $\nu$, $\rho$, $f$, $I(f)$ and $I_n(f)$ be as above. Let $Y$ be an $\mx$-valued random variable with law $\nu$. Let $L = D(\nu||\mu)$ be the Kullback--Leibler divergence of $\mu$ from $\nu$, that is,
\[
L= D(\nu||\mu) = \int_{\mx} \rho(x)\log \rho(x) d\mu(x) = \int_{\mx} \log \rho(y) d\nu(y) = \ee(\log \rho(Y))\,.
\]
Let $\|f\|_{L^2(\nu)} := (\ee(f(Y)^2))^{1/2}$. If $n = \exp(L + t)$ for some $t\ge 0$, then 
\[
\ee|I_n(f)-I(f)| \le \|f\|_{L^2(\nu)} \bigl(e^{-t/4} + 2\sqrt{\pp(\log \rho(Y) > L + t/2)}\bigr)\,.
\]
Conversely, let $1$ denote the function from $\mx$ into $\rr$ that is identically equal to $1$. If $n= \exp(L-t)$ for some $t\ge 0$, then for any $\delta\in (0,1)$,
\[
\pp(I_n(1)\ge 1-\delta) \le e^{-t/2} + \frac{\pp(\log \rho(Y)\le L-t/2)}{1-\delta}\,.
\]
\end{thm}
Note that Theorem \ref{impthm} does not just give the sample size required to ensure that $I_n(f)$ is close to $I(f)$ in the $L^1$ sense; the second part of the theorem implies that if we are below the sample size prescribed by Theorem~\ref{impthm}, then for $f\equiv 1$, there is a substantial chance that $I_n(f)$ is actually {\it not close} to $I(f)$. Such lower bounds cannot be given merely by moment estimates. For example, lower bounds on moments like $\ee|I_n(f)-I(f)|$ and $\var(I_n(f))$ imply nothing; $I_n(f)$ may be close to $I(f)$ with high probability and yet $\ee|I_n(f)-I(f)|$ and $\var(I_n(f))$ may be large. The second part of Theorem \ref{impthm} gives an actual lower bound on the sample size required to ensure that $I_n(f)$ is close to $I(f)$ with high probability, and the first part shows that this lower bound matches a corresponding upper bound. It is interesting that the sample size required for small $L^1$ error turns out to be the actual correct sample size for good performance. 

As shown later in this section, it is fairly common that $\log \rho(Y)$ is concentrated around its expected value in large systems.  In this situation, a sample of size roughly $\exp(D(\nu||\mu))$ is both necessary and sufficient. 

The second main result of this article, stated below, gives the analogous result for the estimate $J_n(f)$. The conclusion is essentially the same. 
\begin{thm}\label{selfimpthm}
Let all notation be as in Theorem \ref{impthm} and let $J_n(f)$ be the estimate defined in \eqref{jndef}. Suppose that $n= \exp(L+t)$ for some $t\ge 0$. Let
\[
\ep := \bigl(e^{-t/4} + 2\sqrt{\pp(\log \rho(Y) > L+t/2)}\bigr)^{1/2}\,.
\]
Then 
\[
\pp\biggl(|J_n(f)- I(f)|\ge \frac{2\|f\|_{L^2(\nu)}\ep}{1-\ep}\biggr)\le 2\ep\,.
\]
Conversely, suppose that $n= \exp(L-t)$ for some $t\ge 0$. Let $f(x)$ denote the function from $\mx$ into $\rr$ that is equal to $1$ when $\log \rho(x)\le L-t/2$ and $0$ otherwise. Then $I(f)=\pp(\log \rho(Y)\le L-t/2)$ but $\pp(J_n(f) \ne 1) \le e^{-t/2}$. 
\end{thm}
Sometimes, importance sampling is used to estimate the probabilities of rare events under the target measure $\nu$. Typically, the quantity of interest is $\nu(A)$, where $A$ is a rare event under $\nu$ but is not a rare event under $\mu$. The method of estimation is the same as before, that is, let $1_A(x)$ be the function that is $1$ if $x\in A$ and $0$ otherwise, and let $I_n(1_A)$ be the importance sampling estimate of $\nu(A)$. The difference with the previous setting is that when estimating $\nu(A)$, we are not satisfied if $|I_n(1_A)-\nu(A)|$ is small because $\nu(A)$ itself is a small number. Rather, it is satisfactory if the ratio $I_n(1_A)/\nu(A)$ is close to $1$. It turns out that the sample size that is necessary and sufficient for this purpose is not $\exp(D(\nu||\mu))$, but $\exp(D(\nu_A||\mu))$, where $\nu_A$ is the probability measure $\nu$ conditioned on the event $A$. This is quantified by the following theorem, which is the third main result of this paper.
\begin{thm}\label{rarethm}
Let all notation be as in Theorem \ref{impthm}. Let $A$ be any event such that $\nu(A) > 0$ and let $1_A$ be the indicator function of $A$, defined above. Let $\nu_A$ be the measure $\nu$ conditioned on the event $A$, that is, for any event~$B$,
\[
\nu_A(B) := \frac{\nu(A\cap B)}{\nu(A)}\,.
\]
Let $\rho_A(x) := \rho(x)1_A(x)/\nu(A)$ be the probability density function of $\nu_A$ with respect to $\mu$. Let $L_A := D(\nu_A||\mu)$. If $n = \exp(L_A + t)$ for some $t\ge 0$, then
\[
\ee\biggl|\frac{I_n(1_A)}{\nu(A)}-1\biggr|\le e^{-t/4} + 2\sqrt{\pp(\log\rho_A(Y)> L_A + t/2\mid Y\in A)} \,.
\]
Conversely, suppose that $n = \exp(L_A-t)$ for some $t\ge 0$. Then for any $\delta\in (0,1)$,
\[
\pp\biggl(\frac{I_n(1_A)}{\nu(A)} \ge 1-\delta\biggr) \le e^{-t/2} + \frac{\pp(\log \rho_A(Y)\le L_A - t/2\mid Y\in A)}{1-\delta}\,.
\]
\end{thm}
We would like to remark here that the upper bounds in Theorems \ref{impthm}, \ref{selfimpthm} and \ref{rarethm} may not be tight. The only purpose of these theorems is to give matching upper and lower bounds on the sample size required for good performance of importance sampling. No attempt was made to get optimal error bounds, especially of the type that is relevant to practitioners.

Another remark is that in practice, $\mu$ is chosen depending on $\nu$, to minimize the required sample size. One potential use for our theorems is that they may be used to choose $\mu$ by minimizing the Kullback--Leibler divergence of $\mu$ from $\nu$ among some class of candidate measures. This point is elaborated in the literature review at the end of this section. 

Sometimes, however, $\mu$ is chosen depending on both $\nu$ and $f$. Since  Theorems~\ref{impthm} and \ref{selfimpthm} give bounds that depend only on the $L^2(\nu)$ norm of $f$, they will not be useful for choosing $\mu$ using fine properties of $f$. This is particularly problematic if $f$ is something like the indicator of a rare event. This issue is partially addressed in Theorem \ref{rarethm}, where $f=1_A$ for some rare event $A$, and the required sample size depends on $\mu$, $\nu$ and the event $A$. Therefore Theorem \ref{rarethm} can be used for choosing $\mu$ depending on properties of both $\nu$ and $f$.

Let us now investigate the implications of our theorems in a few simple examples. More complex examples are given in later sections.

\begin{ex}[Binomial distributions]\label{hypothetical}
Let $\mu = \text{Binomial}(N, p)$ and $\nu = \text{Binomial}(N, r)$, where $r> p$. Then 
\[
\log \rho(x) =  x\log\frac{r}{p} + (N-x)\log\frac{1-r}{1-p}\,.
\]
Let $Y\sim \nu$. Then $L = \ee(\log \rho(Y)) = N H(r,p)$, where
\[
 H(r,p) = r\log\frac{r}{p} + (1-r)\log\frac{1-r}{1-p}\,.
\]
Moreover, the standard deviation of $\log \rho(Y)$ is of order $\sqrt{N}$. Thus, the required sample size is $\exp(N H(r,p) + O(\sqrt{N}))$. On the other hand, a simple calculation shows that if variance is used to determine sample size, the required size would be $\exp(N \, V(r,p))$, where
\[
V(r,p) = \log\biggl(\frac{r^2}{p}+\frac{(1-r)^2}{1-p}\biggr)\,.
\]
By Jensen's inequality, $V(r,p) \ge H(r,p)$. Figure \ref{fig0} shows that graph of $H(r,p)$ versus the graph of $V(r,p)$, as $r$ varies and $p$ is fixed at $1/2$. This elementary example demonstrates how using the variance can lead to unnecessarily large sample sizes.
\end{ex}
\begin{figure}
\includegraphics[width = .7\textwidth]{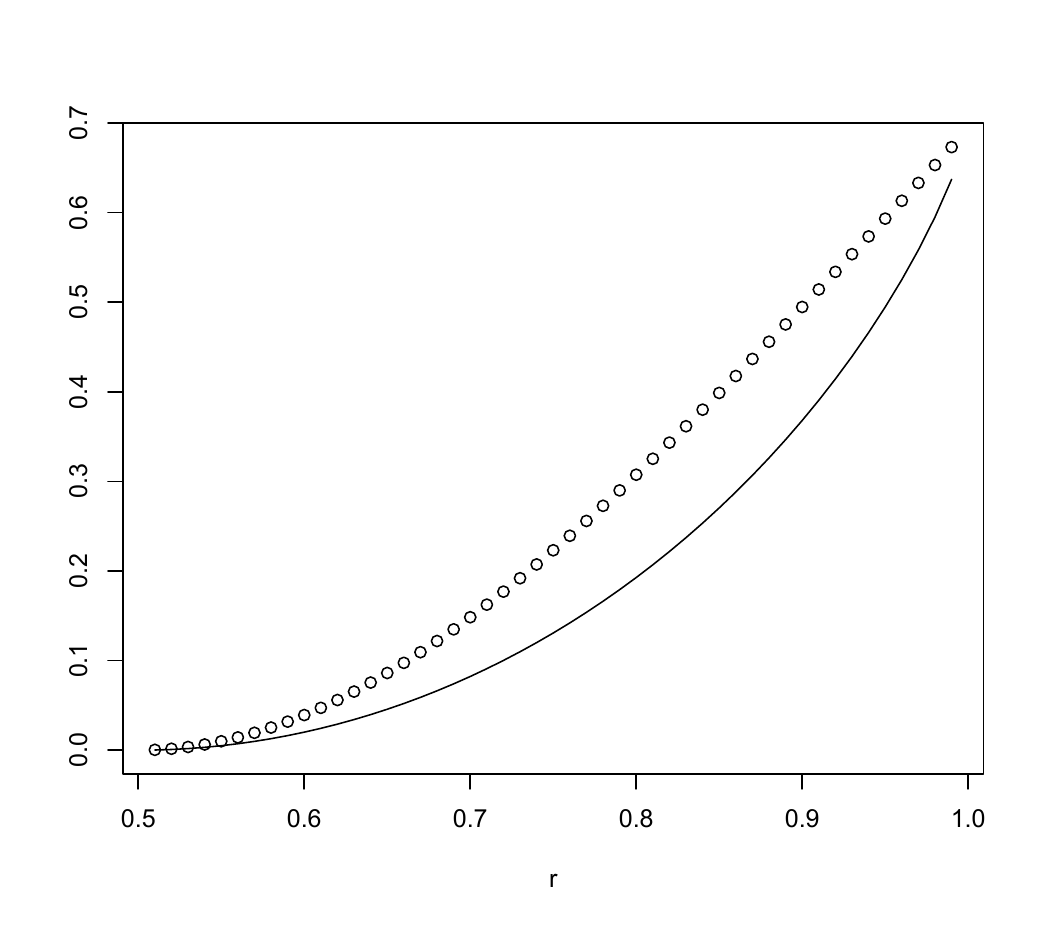}
\caption{Let $H$ and $V$ be as in Example~\ref{hypothetical}. The dotted line represents $V(r,p)$ and the solid line represents $H(r,p)$. Here $p= 0.5$ and $r$ goes from $0.5$ to $1$ on the $x$-axis.}
\label{fig0}
\end{figure}

\begin{ex}[Directed paths]
Let $\mx$ be the set of all monotone paths from $(0,0)$ to $(n,n)$ in the two dimensional lattice. Here, paths are only allowed to go up and to the right. The target measure is the uniform distribution on all such paths. Clearly, $|\mx| = {2n \choose n}$. The sampling measure $\mu$ in this example constructs a random path $\gamma$ as follows (this is known as sequential importance sampling): Choose one of the two directions `up' or `right' with probability $1/2$ until the walk hits the top or right side of the $n\times n$ `box', when the remainder of the walk is forced. If $T(\gamma)$ is the first time the path hits the top or right side then 
\[
\mu(\gamma)=2^{-T(\gamma)}\,.
\]
Both the uniform distribution $\nu(\gamma)=1/{2n\choose n}$ and $\mu(\gamma)$ have the property that, conditional on $T(\gamma)=j$, the paths are uniformly distributed. Thus distributional questions are determined by the distribution of $T(\gamma)$. 

The following proposition from \citet{bassettidiaconis06} shows that under the sampling distribution $\mu$, $T(\gamma)$ is usually about $O(\sqrt{n})$ from the maximum possible $2n-1$, but under the uniform distribution $\nu$, $T(\gamma)$ is usually about $O(1)$ away from $2n-1$. 
\begin{prop}
With the notation above,
\begin{enumerate}
\item[\textup{(a)}] Under the importance sampling distribution $\mu$, 
\[
\mu\{T(\gamma)=j\} = 2^{1-j} {j-1 \choose n-1}, \ \ n\le j\le 2n-1\,.
\]
\item[\textup{(b)}] For $n$ large and fixed positive $x$,
\[
\mu\biggl\{\frac{2n-1-T(\gamma)}{\sqrt{n}} \le x\biggr\} \sim \frac{1}{\pi}\int_0^x e^{-y^2/4}dy\,.
\]
\item[\textup{(c)}] Under the uniform distribution $\nu$, 
\[
\nu\{T(\gamma)=j\} = \frac{2{j-1\choose n-1}}{{2n\choose n}}, \ \  n\le j\le 2n-1\,.
\]
Further $\ee_\nu(T(\gamma)) = (2-\frac{2}{n+1}) n$. 
\item[\textup{(d)}] For $n$ large and any fixed $k$,
\[
\nu\{T(\gamma)=2n-1-k\} \sim \frac{1}{2^{k+1}}, \ \ 0\le k<\infty\,.
\]
\end{enumerate}
\end{prop}
The quantity $L$ of Theorem \ref{impthm} is determined from $\rho(\gamma)=\nu(\gamma)/\mu(\gamma)$ as
\begin{align*}
L &= \sum_\gamma \nu(\gamma) \log\frac{\nu(\gamma)}{\mu(\gamma)} \\
&= -\log {2n\choose n} + \frac{\log 2}{{2n \choose n}} \sum_\gamma T(\gamma) \\
&= -\log {2n\choose n} + \biggl(2-\frac{2}{n+1}\biggr) n \log 2\\
&= \log \sqrt{\pi n} - 2\log 2 + O\biggl(\frac{1}{n}\biggr)\,.
\end{align*}
Thus, $e^L \sim \sqrt{\pi n}/4$, and moreover, $\log \rho(\gamma)$ has fluctuations of order $1$ around its mean. Thus, a sample size of order $\sqrt{n}$ is necessary and sufficient for accuracy of importance sampling in this example. The sufficiency was already observed using variance computations in \citet{bassettidiaconis06}; the necessity is a new result. Similar computations can be carried out for paths allowed to go left or right or up (staying self avoiding) using results of \citet{bousquetmelou14}. 
\end{ex}

\begin{ex}[Estimating the probability of a rare event]\label{thm3ex}
As an example for Theorem \ref{rarethm}, fix $N$ and $p>1/2$ and let $A = \{j: Np\le j\le N\}$. Take $\nu$ to be the Binomial$(N,1/2)$ distribution. Let $b(A;N, 1/2)$ be the probability of $A$ under $\nu$. 
Estimating $b(A; N,1/2)$ by simple sampling from $\nu$ would be a crazy task; for example when $N=100$ and $p=.9$, $b(A;100, 1/2) \approx 0.676049\times 10^{-45}$, which means that we would need roughly $10^{45}$ samples to directly estimate this probability.  A standard importance sampling approach (\citet{siegmund76}) is to sample $X_1,X_2,\ldots,X_n$ from $\mu = \text{Binomial}(N, \theta)$ for some $\theta$ and use 
\[
I_n(A)= \frac{1}{n}\sum_{i=1}^n \frac{\nu(X_i)}{\mu(X_i)} 1_A(X_i)\,.
\]
Theorem \ref{rarethm} shows that this will be accurate in ratio for $n$ of order $e^{L_A}$. The following proposition shows that when $\mu$ is Binomial$(N, \theta)$,  $\theta = p$ minimizes $e^{L_A}$, agreeing with the variance minimization in  \citet{siegmund76}.  When $N=100$ and $p=.9$, $e^{L_A} \approx 1.723 \times 10^{28}$ (still an impossible sample size, but much smaller than $10^{45}$).
\begin{prop}\label{rareprop}
Fix $N$ and $p>1/2$ such that $Np$ is an integer. Let $\mu$ be the \textup{Binomial}$(N, \theta)$ distribution, $\nu$ be the \textup{Binomial}$(N, 1/2)$ distribution and $A = \{j:Np\le j\le N\}$. Then the quantity $L_A$ of Theorem~\ref{rarethm} is asymptotically minimized when $\theta=p$, and with this choice of $\theta$, $L_A$ is aymptotic to $-2N \log(p^p(1-p)^{1-p})$ 
as $N\to \infty$.
\end{prop}
\end{ex}

\vskip.2in
\noindent {\bf Review of the literature.} 
Our interest in this topic started with a question from our colleague Don Knuth in \citet{knuth76}. He used sequential importance sampling to generate random self-avoiding paths starting at $(0,0)$ and ending at $(N,N)$ in a two dimensional $N\times N$ grid. For $N=10$ he calculated the number of paths (about $1.6\times 10^{24}$), the average path length ($92\pm 5$) and the proportion of paths passing through $(5,5)$ ($81\%\pm 10\%$). He noticed huge fluctuations along the way and wanted to know about the accuracy of his estimates. In the follow up work \citet{knuth96}, exact computation showed surprising accuracy for his example. \citet{bassettidiaconis06} and \citet{bousquetmelou14} studied toy versions of Knuth's problem where exact calculations can be done; they confirm the extreme variability and make the accuracy observed mysterious.

In our work,  the choice of the proposal measure $\mu$ is considered fixed. A good deal of the art of successful implementation of importance sampling consists in a careful choice of $\mu$, adapted to the problem under study. This is often done to minimize the variance of the resulting estimate. Our work, especially the main result of Section \ref{problem}, suggests that the variance is a poor measure of accuracy for these long tailed problems. Thus, there is work to be done, exploring ways of adapting the many good ideas below, based on the variance, to minimizing the Kullback--Leibler divergence. 

Any book on simulation will treat importance sampling. We recommend~\citet{hammersleyhandscomb65},  \citet{srinivasan02}, \citet{cappeetal05} and \citet{liu08}. To begin our review of the research literature, a classical choice of the sampling measure $\mu$ for estimating $I(f) = \int fd\nu$ is to take $d\mu(x)$ proportional to $|f(x)|d\nu(x)$ (\citet{kahnmarshall53}). \citet{hesterberg95} suggests using a mixture of measures for $\mu$ with one component proportional to $|f(x)|$ near its maximum. This is closely related to the widely used method of umbrella sampling (\citet{torrievalleau77}; nicely developed in \citet{madras98}). \citet{owenzhou00} combine Hesterberg's idea with control variates to give an attractive, practical approach. In later work, \citet{owenzhou99} suggest an adaptive version, attempting to improve the proposed $\mu$ using previous sampling. This is based on the empirical variance which means that our laments in Section \ref{problem} apply.

The idea of using $L^1$ distance to measure performance of importance sampling has appeared in a few prior instances. Two notable examples are \citet{owen05} and \citet{owen06}, where $L^1$ error was used to compare the Monte Carlo and quasi-Monte Carlo approaches to estimating singular integrands via importance sampling.

Importance sampling is often used to do rare event simulation. Then, it is natural to tilt the sampling distribution $\mu$ towards to the region of interest. \citet{siegmund76} gives an asymptotically principled approach to doing this, which has given rise to much follow-up work, some of it quite deep mathematically. A  unifying account of a variety of importance sampling algorithms for simulating the maxima of Gaussian fields appears in \citet{shietal}. A host of novel ways of building importance sampling estimates for problems such as estimating the size of the union of a collection of sets when the size of each is known is in \citet{naimanwynn97}. The work of Paul Dupuis with many coauthors is notable here. \citet{dupuiswang04} and \citet{dupuisetal12} are representative papers with useful pointers to an extensive literature. \citet{asmussenglynn07} give a textbook account of this part of the subject. 

An important part of the literature adapts importance sampling from the case of independent proposals considered here to use with a Markov chain generating proposals. \citet{madraspiccioni99} give a clear development as do the textbook accounts of \citet{robertcasella04}  or \citet{liu08}.

An important class of techniques for building proposal distributions is known as sequential importance sampling. An early appearance of this to sampling self-avoiding paths occurs in \citet{rosenbluth55}. For contingency table examples see \citet{chenetal05}. For degree sequences of graphs, see \citet{blitzsteindiaconis10}.  For time series and a general review see the textbook by \citet{doucetetal01} or the survey of \citet{chenliu07}.

A relatively recent technique choosing the proposal distribution, which has been particularly successful in the heavy-tailed setting, is a method based on Lyapunov functions developed by \citet{blanchetliu08, blanchetliu10}, \citet{blanchetglynn08} and \citet{blanchetglynnleder12}.

One large related topic is the connection between importance sampling and particle filters. Roughly, when building a proposal $\mu$ sequentially, one begins with a number $N$ of starts. As the proposals are independently built up, some weights may be much larger than others. One can generate $N$ new proposals from the present ones (say with probability proportional to weights). This will replicate some proposals and kill of those with smaller weights. This resampling can be repeated several times. The final weighted samples are used, in the usual way, to form importance sampling estimates. This large enterprise can be surveyed in the textbooks of \citet{delmoral04, delmoral13} and \citet{doucetetal01}.  Work of \citet{chanlai07, chanlai11} harnesses martingale central limit theorems to get the limiting distribution of these importance sampling methods in a variety of complex stochastic models. The web page of Arnaud Doucet is extremely useful. A very clear recent paper is: \citet{delmoraletal15}.

Besides the broad classifications outlined above, importance sampling has a variety of other applications that are harder to categorize. A recent example is the paper by \citet{efron12} that suggests the use of importance sampling for generating from Bayesian posterior distributions. In this context, an interesting note is that simulating from a Bayesian posterior by rejection sampling was investigated by \citet{fmr10}, who found a connection with the Kullback--Leibler divergence that bears some similarities with the results of this paper. 

Two other recent papers have similarities with our work. One is that of \citet{hultnyquist16}, who analyze the performance of importance sampling in the estimation of probabilities of rare events using large deviation techniques. The Kullback--Leibler divergence arises naturally in this work, due to its appearance in large deviation rate functions. The other is a paper of~\citet{agapiouetal15}, who prove that $|I_n(f)-I(f)|$ is small if $n\ge \ee(\rho(Y))\ge e^L$, in the notation of our Theorem~\ref{impthm}. This result is applied to a class of problems that don't overlap with our set of examples, making \cite{agapiouetal15} and this paper complementary to each other.

\section{Testing for convergence}\label{problem} 
The theory developed in Section \ref{theory}, while theoretically interesting, is possibly not very useful from a practical point of view. Determining $D(\nu||\mu)$ requires in-depth knowledge of not only the measure $\mu$, but also the usually much more complicated measure $\nu$. It is precisely the lack of understanding about $\nu$ that motivates importance sampling, so it seems pointless to ask a practitioner to compute the required sample size by using properties of $\nu$. 

To determine whether the importance sampling estimate has converged, a common practice is to estimate $\var(I_n(f))$ by estimating the variance formula~\eqref{varform} using the data from $\mu$. One natural estimate is
\[
v_n(f) := \frac{1}{n^2}\sum_{i=1}^n f(X_i)^2 \rho(X_i)^2 - \frac{I_n(f)^2}{n}\,.
\]
If this estimate is used, then importance sampling is declared to have converged if for some $n$, $v_n(f)$ turns out to be smaller than some pre-specified tolerance threshold $\ep$ (see \citet{robertcasella04}). 


The following theorem shows that using $v_n(f)$ as a diagnostic for convergence of importance sampling is problematic, because  for any given tolerance level $\ep$, there is high probability that the test declares convergence at or before a sample size that depends only on $\ep$ and not on $\mu$, $\nu$ or $f$. This is absurd, since convergence may take arbitrarily long, depending on the problem. 
\begin{thm}\label{flawthm}
Given any $\ep>0$, there exists $n\le \ep^{-2}2^{1+\ep^{-3}}$ such that the following is true. Take any $\mu$ and $\nu$ as in Theorem \ref{impthm}, and any $f:\mx \ra \rr$ such that $\|f\|_{L^2(\nu)}\le 1$. Let $v_n(f)$ be defined as above. Then $\pp(v_n(f)< \ep) \ge 1-4\ep$. 
\end{thm}
Although the upper bound on $n$ is very large --- for example, for $\ep = .1$ the upper bound is roughly $2.14 \times 10^{303}$ --- Theorem \ref{flawthm} gives a conceptual proof that using $v_n(f)$ for testing convergence of importance sampling is fundamentally flawed. As the measures $\mu$ and $\nu$ get more and more singular with respect to each other (which often happens as system size gets larger), importance sampling should take longer to converge. A test that does not respect this feature cannot be a plausible test for convergence. Incidentally, it is not clear whether the upper bound on $n$ in Theorem \ref{flawthm} can be improved to something more reasonable. 

The ineffectiveness of the variance diagnostic is not hard to demonstrate in examples. One such examples are given below. 
\begin{ex}\label{binex1}
In Example \ref{hypothetical} with large $N$, $v_n(f)$ stays extremely close to zero for any realistic value of $n$ because $\rho(X_i)$ is very close to zero with high probability. But here we know that the actual convergence takes place at a sample size that is exponentially large in $N$. For instance, consider $\mu = $ Binomial$(100,.5)$ and $\nu = $ Binomial$(100,.7)$. Let $f$ be the function that is identically equal to $1$.  Figure \ref{fig3} shows the plot of the estimated standard deviation $\sqrt{v_n(f)}$ against $n$, as $n$ ranges from $1$ to $10^6$. The estimated standard deviation remains fairly small throughout. However, since we know the actual value of $I(f)$ in this case (which is $1$), it is easy to compute the actual error $|I_n(f)-I(f)|$ and check that the variance diagnostic is giving a false conclusion. 
\end{ex}
\begin{figure}[t]
\includegraphics[scale = .6]{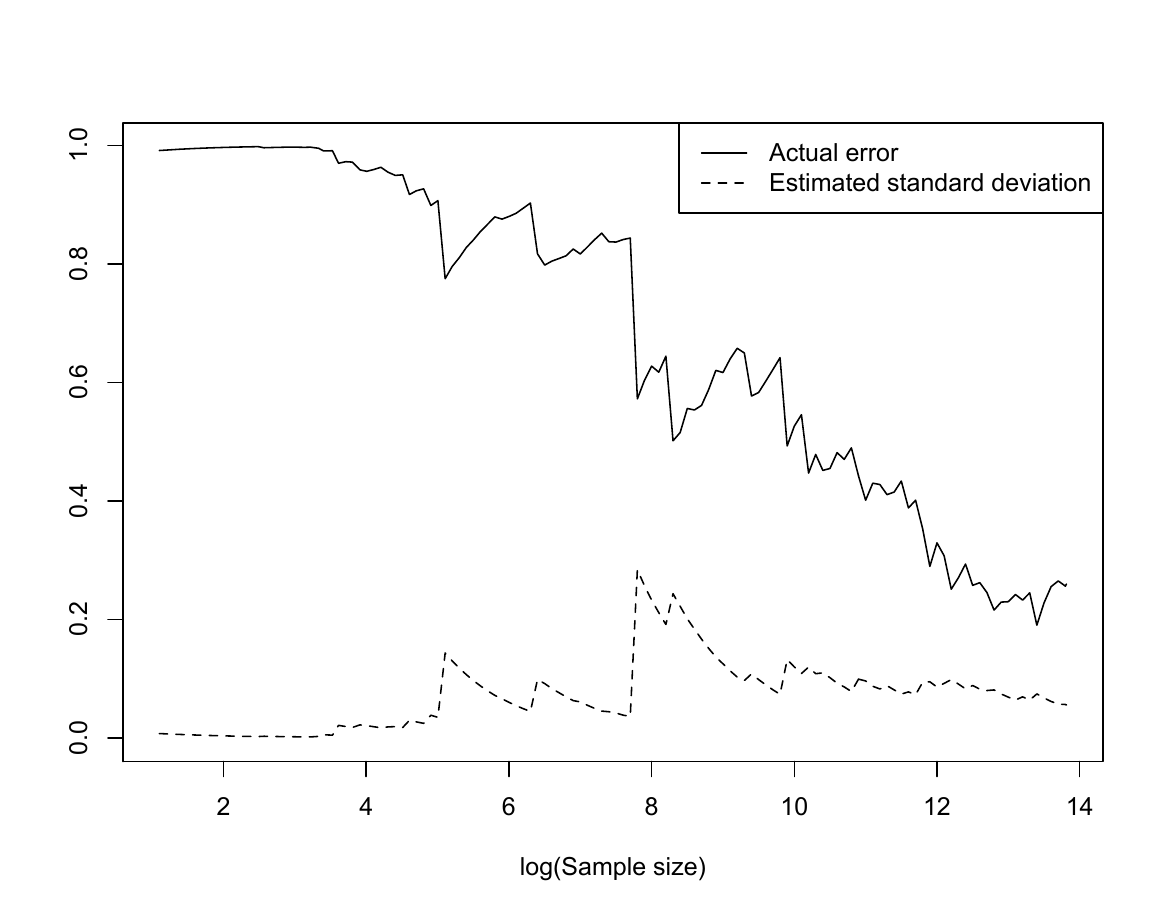}
\caption{Estimated standard deviation $\sqrt{v_n(f)}$ of $I_n(f)$, as $n$ ranges from $1$ to $10^6$, and the actual error $|I_n(f)-I(f)|$. Here $\mu = $ Binomial$(100,.5)$, $\nu = $ Binomial$(100,.7)$ and $f\equiv 1$.} 
\label{fig3}
\end{figure}
There are results in the literature that claim to show that the variance estimation method gives a valid criterion for the convergence of importance sampling. However, what these results actually show is that if $n$ is so large that the importance sampling estimates are accurate, then $v_n(f)$ is small. In other words, the smallness of $v_n(f)$ is a {\it necessary} condition for convergence of importance sampling, but not a {\it sufficient} condition. For a diagnostic criterion to be useful, it needs to be both necessary and sufficient for convergence.

In practice, $v_n(f)$ is not usually the preferred diagnostic. Various self-normalized versions of $v_n(f)$ are used. It is possible that these more complicated estimates are also problematic in the same way, but we do not have a proof. It would be interesting to prove analogs of Theorem \ref{flawthm} for self-normalized diagnostic statistics.

In view of Theorem \ref{impthm}, it is natural to consider estimates of the Kullback--Leibler divergence as possible diagnostic tools for convergence. However, an inspection of the proof of Theorem \ref{flawthm} indicates that such estimates are likely to suffer from similar problems. The issue is that any diagnostic criterion that is itself dependent on the accuracy of an estimate obtained by importance sampling, is unlikely to be effective as a measure of the efficacy of importance sampling.

We suggest the following alternative diagnostic that is not itself an importance sampling estimate of any quantity. As usual, let $\mu$ be the sampling measure, $\nu$ be the target measure, and $\rho = d\nu/d\mu$. Let $X_1, X_2,\ldots$ be i.i.d.~random variables with law $\mu$. Define $q_n := \ee(Q_n)$, where
\[
Q_n := \frac{\max_{1\le i\le n} \rho(X_i)}{\sum_{i=1}^n \rho(X_i)}\,.
\]
The size of $q_n$ is our criterion for diagnosing convergence of importance sampling. The general prescription is that if for some value of $n$ the quantity $q_n$ is smaller than some pre-specified threshold (say, $0.01$), declare that $n$ is large enough for importance sampling to work. Note that the random variable $Q_n$ always lies between $0$ and $1$, and therefore $q_n\in [0,1]$. Moreover, given any $n$, it is possible to estimate $q_n$ up to any desired degree of accuracy by repeatedly simulating $Q_n$ and taking an average, since $q_n = \ee(Q_n)$ and $Q_n$ always lies between $0$ and $1$. Lastly, note that for estimating $q_n$ using simulations in the above manner, it suffices to know the density $\rho$ up to an unspecified normalizing constant. Repeatedly calculating $Q_n$, however, may be computationally expensive if either $n$ is too large or $\rho$ is too complex.

Why should one expect the smallness of the  quantity $q_n$ to be a valid diagnostic criterion for convergence of importance sampling? First, let us hasten to add the caveat that one can produce examples where it does not work. One such example is the following: Take a large number $N$. Let $\mu$ be the uniform distribution on $\{1,2,\ldots, N\}$. Let $\nu$ be the distribution that puts mass $1/2N$ on the points $1,2,\ldots, N-1$, and mass $(N+1)/2N$ on the point $N$. Then $\rho(x) = 1/2$ for $x=1,2,\ldots, N-1$ and $\rho(N)=(N+1)/2$. Under the sampling measure $\mu$, $\rho = 1/2$ with probability $1-1/N$. Therefore when $1\ll n\ll N$, the quantity $q_n$ will be small; but convergence of importance sampling will not happen until $n\gg N$. 

In spite of the above counterexample, we expect that $q_n$ is a valid diagnostic for many natural examples. This is made precise to a certain extent in the setting of Gibbs measures by Theorem~\ref{newthm} in the next section. A general heuristic argument for the effectiveness of the $q_n$ diagnostic, on which the proof of Theorem \ref{newthm} is based, can be described as follows. 

Suppose that $\log \rho$ is concentrated under $\nu$, so that Theorem \ref{impthm} applies, and the sample size required for convergence of importance sampling is roughly $e^L$, where $L = \ee_\nu(\log \rho)$. Take any $n$ below this threshold. Let $M_n:= \max_{1\le i\le n}\rho(X_i)$. Since $\rho(X_1),\rho(X_2),\ldots$ are i.i.d.~random variables, it is easy to see that under mild conditions, $M_n \approx a$ with high probability, where $a$ solves
\begin{align}\label{heur0}
n \pp(\rho(X_1)\ge a) = 1\,.
\end{align}
Next, let $S_n := \sum_{i=1}^n \rho(X_i)$. Since $M_n \approx a$, therefore
\[
S_n \approx \sum_{i=1}^n \rho(X_i) 1_{\{\rho(X_i)\le a\}}\,.
\]
Therefore
\begin{align}\label{heur1}
\ee(S_n) \approx n \ee(\rho(X_1)1_{\{\rho(X_1)\le a\}}) = n \pp_\nu(\rho \le a) = n \pp_\nu(\log \rho \le \log a)\,.
\end{align}
Now, $\ee_\nu(\log \rho)=L > \log a$. Thus, $\pp_\nu(\log \rho \le \log a)$   is a large deviation probability. Therefore under mild conditions, one may expect that
\[
\pp_\nu(\log \rho \le \log a) \approx \pp_\nu(\log \rho \approx \log a)\,.
\]
Plugging this into \eqref{heur1}, we get
\begin{align*}
\ee(S_n) &\approx n\pp_\nu(\log \rho \approx \log a)\\
&=  n \ee(\rho(X_1)1_{\{\rho(X_1)\approx a\}})\\
&= na \pp(\rho(X_1)\approx a) \le na \pp(\rho(X_1)\ge a)\,.
\end{align*}
Using the equation \eqref{heur0} to evaluate the last term, we get $\ee(S_n) \lesssim a$, and therefore $S_n = O(a)$ by Markov's inequality. Since $M_n \approx a$, this shows that
\[
q_n = \ee\biggl(\frac{M_n}{S_n}\biggr) = \Omega(1)\,,
\]
where $\Omega(1)$ means a quantity that is uniformly bounded away from zero as $n\to \infty$. 
The above heuristic shows that if $n \ll e^L$ and some appropriate conditions hold, then $q_n = \Omega(1)$. In other words, smallness of $q_n$ should be a sufficient condition for convergence of importance sampling. This sketch can be made rigorous under certain circumstances. An instance of this is illustrated by Theorem \ref{newthm} in the next section.

The smallness of $q_n$ is also a necessary condition for convergence of importance sampling. Unlike sufficiency, the necessity can be rigorously proved in full generality.
\begin{thm}\label{qthm}
Let all notation be as in Theorem \ref{impthm}. Let $q_n$ be defined as above. Let $\ep_n := \ee|I_n(1)-1|$. Then 
\[
q_n \le C\max\biggl\{\frac{1}{n},\,\frac{\log \log (1/\ep_n)}{\log(1/\ep_n)}\biggr\}\,,
\]
where $C$ is a universal constant. 
\end{thm}
As mentioned above, this theorem shows that the smallness of $q_n$ is a necessary condition for convergence of importance sampling (recalling that by Theorem \ref{impthm}, convergence in $L^1$ is equivalent to actual good performance); if $\ep_n$ is small, then $q_n$ is forced to be small.  This is, however, a conceptual theorem. The bound is too poor to be applicable in practice, and the unspecified universal constant $C$ can also be too large for the theorem to have any practical relevance.

The performance of $q_n$ in Example \ref{binex1} is depicted in Figure \ref{fig4}. The figure plots the estimated standard deviation $\sqrt{v_n(f)}$ and the statistic $q_{n}$, against $\log n$ as $n$ ranges from $1$ to $10^6$. As in Figure \ref{fig3}, we see that the estimated standard error is  generally quite misleading and unstable. On the other hand the statistic $q_{n}$ detects the non-convergence in small samples and is very stable. The estimation of $q_{n}$ was based on a sample of size $500$ for each~$n$. 
\begin{figure}[t]
\includegraphics[scale = .6]{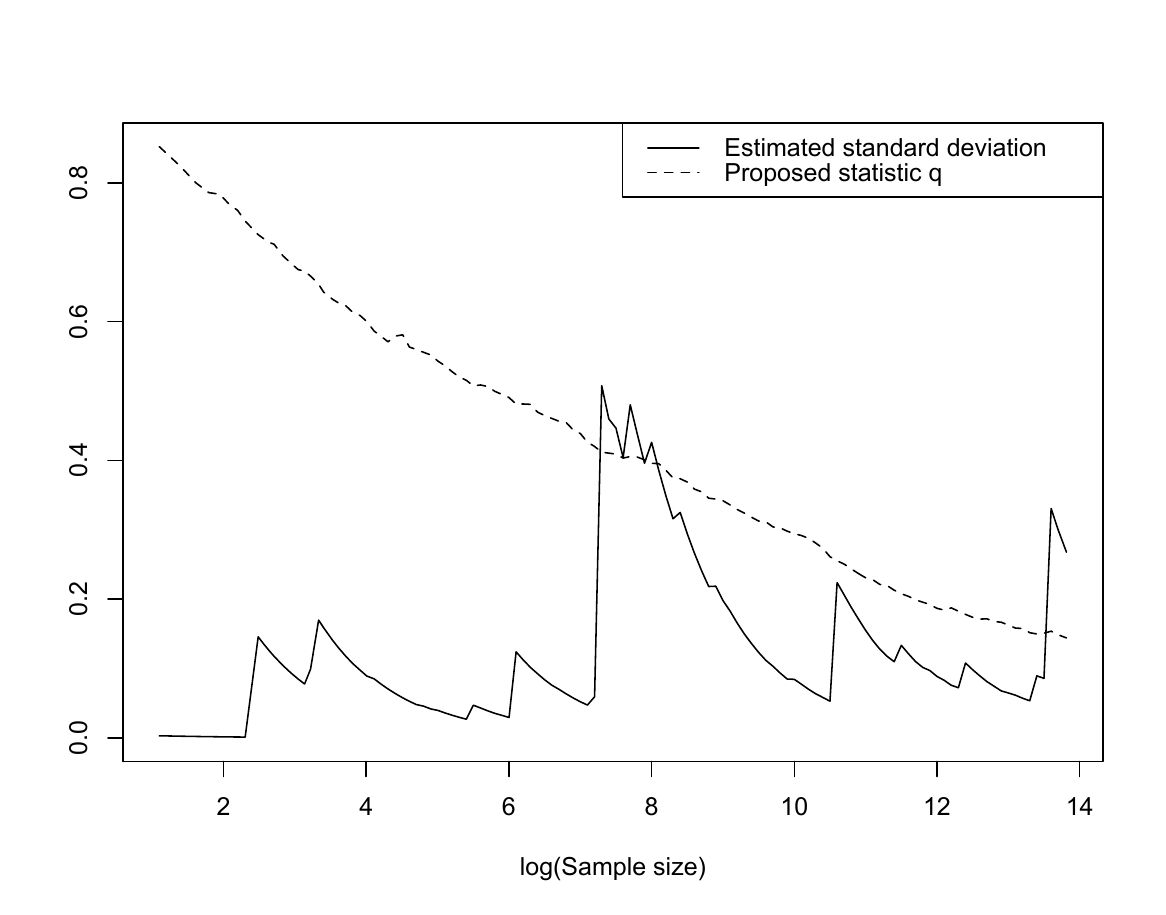}
\caption{Performance of $q_{n}$ in Example \ref{binex1}, plotted against the natural logarithm of the sample size.}
\label{fig4}
\end{figure}

Another illustration is given in Figure \ref{fig5}, which investigates the performance of $q_n$ for Knuth's self-avoiding walks on a $10\times 10$ grid, that was described in the literature review part of Section~\ref{theory}. The plot shows the behavior of $q_n$ as $n$ ranges from $1$ to $10^5$. We see that $q_n$ is not too small (greater than $0.2$) when $n = 10^3$, but starts getting appreciably small around $n = 10^4$. When $n= 10^5$, $q_n$ is minuscule.
\begin{figure}[t]
\includegraphics[scale = .4]{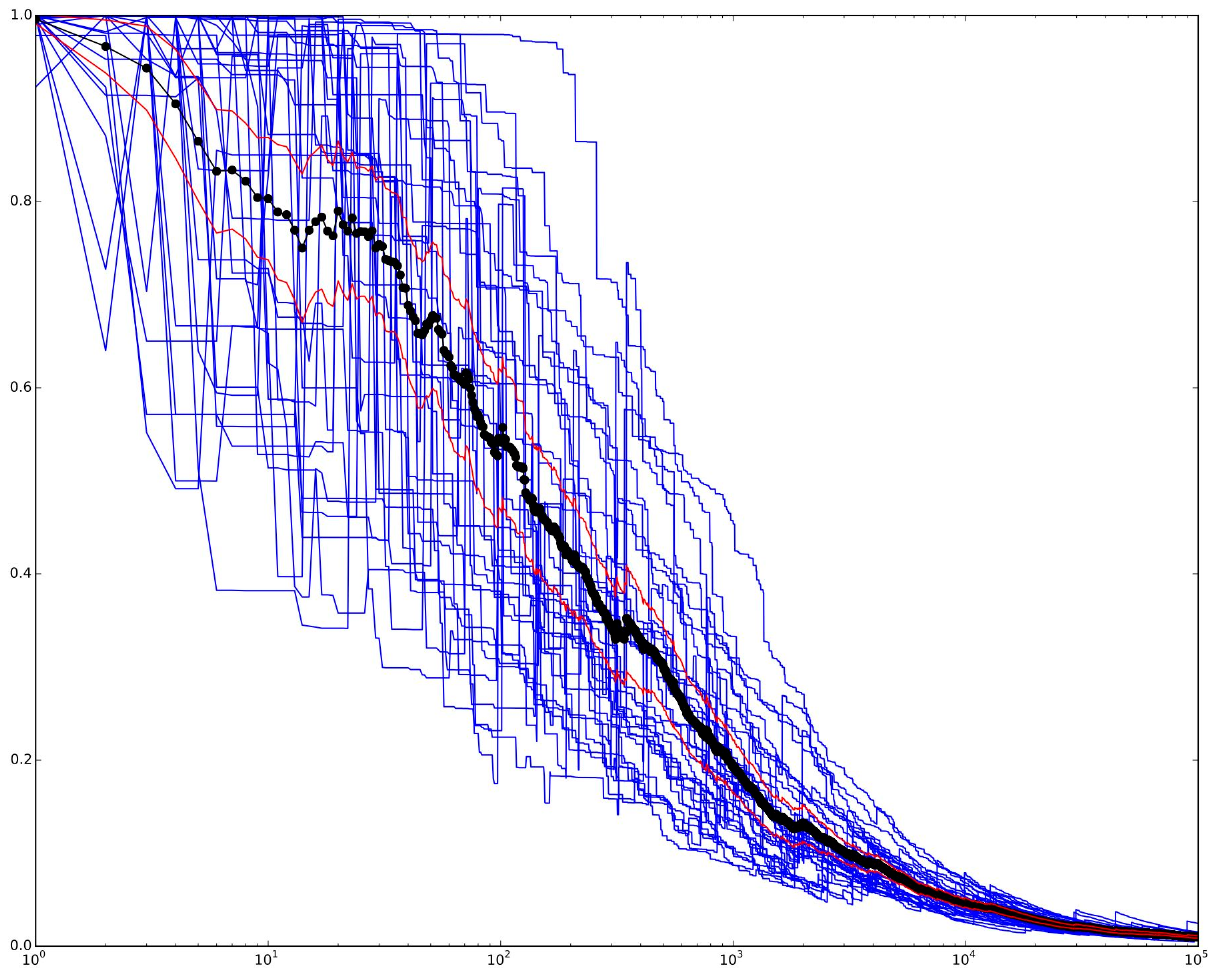}
\caption{Performance of $q_{n}$ for Knuth's self-avoiding walks on a $10\times 10$ grid. The values of $q_n$, denoted by the thick dots, were estimated from $31$ simulations of $Q_n$, which are depicted by the solid lines. Picture courtesy of Marc Coram.}
\label{fig5}
\end{figure}

The random quantity $Q_n$ is closely related to some existing diagnostics in the literature on sequential Monte Carlo (particle filters). It has the same form as the $\infty$-ESS statistic proposed by \citet{hr15} in the context of sequential Monte Carlo. Here ESS stands for `Effective Sample Size', a familiar concept in the sequential Monte Carlo literature. There is a substantial body of work on the efficacy of the effective sample size as a diagnostic tool, possibly beginning with \citet{liuchen95} and \citet{doucetetal01}. See \citet{wlh16} for some latest results. \citet{hr15}  established similar properties for the $\infty$-ESS. It would be interesting to see whether analogs of these results can be proved for the $Q_n$ and $q_n$ statistics proposed in this section.


\section{Importance sampling for exponential families (Gibbs measures)}\label{statmech}
As in Section \ref{theory}, let $\mx$ be a set equipped with some sigma-algebra. Let $\lambda$ be a finite measure on $\mx$ that we shall call the `base measure'. Let $H:\mx \ra\rr$ be a measurable function, called the Hamiltonian, and let $\beta\in \rr$ be a parameter, called the inverse temperature. The exponential family distribution (Gibbs measure) $G_\beta$ on $\mx$ defined by the sufficient statistic (Hamiltonian) $H$ at a parameter value (inverse temperature) $\beta$ is the probability measure on $\mx$ that has probability density 
\[
Z(\beta)^{-1}\exp(-\beta H(x))
\]
with respect to the base measure $\lambda$, where
\[
Z(\beta) = \int_{\mx} \exp(-\beta H(x)) d\lambda(x)
\] 
is the normalizing constant, which is assumed to be finite. Let
\[
F(\beta) := \log Z(\beta)\,.
\]
In physics parlance, the quantity $-F(\beta)/\beta$ is known as the free energy of the system at inverse temperature~$\beta$.


Often, the normalizing constant $Z(\beta)$ is hard to calculate theoretically. Importance sampling is used to estimate $Z(\beta)$ in a variety of ways. See \citet{gelmanmeng98} for a useful review. \citet{lelievreetal10} show the breadth of this problem. One simple technique: Let $\beta_0$ be an inverse temperature at which we know how to generate a sample from the Gibbs measure. For example $\beta_0=0$ is often a good choice, because $G_0$ is nothing but the base measure $\lambda$ normalized to have total mass one. The goal is to estimate $Z(\beta)$ using a sample from $G_{\beta_0}$. Let $X_1,\ldots, X_n$ be an i.i.d.~sample of size $n$ from $G_{\beta_0}$. The importance sampling estimate of $Z(\beta)$ based on this sample is the following:
\[
\hat{Z}_n(\beta):= \frac{Z(\beta_0)}{n}\sum_{i=1}^n \exp(-(\beta-\beta_0)H(X_i))\,.
\] 
It is easy to see that $\ee(\hat{Z}_n(\beta)) = Z(\beta)$. The question is, how large does $n$ need to be, so that the ratio $\hat{Z}_n(\beta)/Z(\beta)$ is close to $1$ with high probability? 


The following theorem shows that under favorable conditions, a sample of size approximately  $\exp(F(\beta_0)-F(\beta)-(\beta_0-\beta)F'(\beta))$ 
is necessary and sufficient. The proof, given in Section \ref{proofs}, is a simple consequence of Theorem \ref{impthm} since $F(\beta_0)-F(\beta)-(\beta_0-\beta)F'(\beta)$ is actually the Kullback--Leibler divergence of $G_{\beta_0}$ from $G_\beta$. This theorem is a result for finite systems. A more general version of this result that applies in the thermodynamic limit is given later in this section. 
\begin{thm}\label{gibbsthm}
Let all notation be as above. Suppose that the  Hamiltonian $H$ satisfies the condition that for some $\beta' > |\beta|$,
\[
\int_{\mx} \exp(\beta'|H(x)|) d\lambda(x) <\infty\,.
\]
Then $F$ is infinitely differentiable at $\beta$. Let 
\[
L := F(\beta_0) - F(\beta) - (\beta_0-\beta)F'(\beta)
\]
and
\[
\sigma := |\beta_0-\beta| \sqrt{F''(\beta)}\,.
\]
If $n = \exp(L+r\sigma)$ for some $r\ge 0$, then 
\[
\ee\biggl|\frac{\hat{Z}_n(\beta)}{Z(\beta)}-1\biggr|\le e^{-r\sigma/4} + \frac{4}{r}\,.
\]
Conversely, if $n = \exp(L-r\sigma)$ for some $r\ge 0$, then for any $\delta\in (0,1)$,
\[
\pp\biggl(\frac{\hat{Z}_n(\beta)}{Z(\beta)}\ge 1-\delta\biggr)\le e^{-r\sigma/2} + \frac{4}{(1-\delta)r^2}\,.
\]
\end{thm}
It is not difficult to verify by direct calculation that $F''$ is always nonnegative. This implies, in particular, that $F$ is convex. As a consequence of this feature, $L$ and $\sigma$ are also nonnegative.

In standard examples, $F$, $F'$ and $F''$ are all of the same order of magnitude, and the magnitudes are large. Therefore $L$ is large and $\sigma = O(\sqrt{L})$, which implies that the required sample size is concentrated  in the logarithmic scale at $\exp(L + O(\sqrt{L}))$. The situation is illustrated through the following examples.
\begin{ex}[Independent spins]\label{gibbs1}
Take some $N\ge 1$ and let $\mx = \{-1,1\}^N$. Let $\lambda$ be the counting measure on this set, and for $x= (x_1,\ldots, x_N)\in \mx$, let 
\[
H(x) = -\sum_{i=1}^N x_i\,.
\]
The $G_\beta$ is nothing but the joint law of $N$ i.i.d.~random variables that take value $1$ with probability $e^{\beta}/(e^{\beta}+e^{-\beta})$ and $-1$ with probability $e^{-\beta}/(e^\beta+e^{-\beta})$. A simple computation gives $Z(\beta) = 2^N (\cosh\beta)^N$. Therefore
\[
F(\beta) = N \log \cosh \beta + N\log 2\,.
\]
Thus, for any given $\beta_0$ and $\beta$,
\[
L = N\log \frac{\cosh \beta_0}{\cosh\beta} - N(\beta_0-\beta)\tanh\beta\,,
\]
and 
\[
\sigma = 4\sqrt{N} |\beta_0-\beta| \,\textup{sech}\,\beta\,.
\]
Therefore, $L$ is typically of order $N$ and $\sigma$ is typically of order $\sqrt{N}$. 
\end{ex}
\begin{ex}[1D Ising model with periodic boundary]\label{gibbs2}
As in the previous example, let $\mx = \{-1,1\}^N$ and let $\lambda$ be the counting measure on this set. For $x=(x_1,\ldots, x_N)\in \{-1,1\}^N$, let
\[
H(x) = -J\sum_{i=1}^N x_i x_{i+1} - h\sum_{i=1}^N x_i\,,
\]
where $J\ge 0$, $h\in \rr$, and $x_{N+1}$ in the first sum stands for $x_1$. This is the Hamiltonian for the one dimensional Ising model for a system of $N$ spins with periodic boundary. The parameters $J$ and $h$ are traditionally known as the coupling constant and the strength of the external magnetic field. The partition function of this model is easily computed by the transfer matrix method (see \citet{baxter82}): $Z(\beta) = \tr(V(\beta)^N)$, where $V(\beta)$ is the $2\times 2$ matrix
\[
\left[
\begin{array}{cc}
e^{\beta(h+J)}& e^{-\beta J}\\
e^{-\beta J} & e^{-\beta(h-J)}
\end{array}
\right]\,.
\]
In other words, if $\lambda_1(\beta)$ and $\lambda_2(\beta)$ are the two eigenvalues of this matrix (arranged such that $|\lambda_1|\ge |\lambda_2|$), then 
\[
Z(\beta) = \lambda_1(\beta)^N + \lambda_2(\beta)^N\,.
\]
Consequently, 
\[
F(\beta) = \log (\lambda_1(\beta)^N + \lambda_2(\beta)^N)\,.
\]
It is not hard to verify that 
\[
\lambda_1(\beta) = e^{\beta J} \cosh \beta h + \sqrt{e^{2\beta J }(\sinh \beta h)^2 + e^{-2\beta J}}
\]
and  
\[
\lambda_2(\beta) = e^{\beta J} \cosh \beta h - \sqrt{e^{2\beta J }(\sinh \beta h)^2 + e^{-2\beta J}}\,.
\]
Using these formulas it is easy to write down explicit formulas for $L$ and $\sigma$ for any given $\beta$ and $\beta_0$, and compute $a(\beta_0, \beta)$ and $b(\beta_0, \beta)$ such that as $N\ra\infty$, $L\sim N a(\beta_0,\beta)$ and $\sigma \sim \sqrt{N} b (\beta_0,\beta)$. 
\end{ex}

Examples \ref{gibbs1} and \ref{gibbs2} demonstrate how Theorem \ref{gibbsthm} can be applied to calculate the sample size required for importance sampling in statistical mechanical models. However, these examples required exact computations in finite systems, which is rarely possible  in complex models. Our next theorem deals with a generic sequence of models that converge to a  limit. Exact computations are assumed to be possible only in the limit.

Let $\{\mx_N\}_{N\ge 1}$ be a sequence of spaces equipped with sigma-algebras and finite measures $\{\lambda_N\}_{N\ge 1}$. For each $N$ let $H_N:\mx\ra\rr$ be a measurable function, and for each $\beta\in \rr$ let $G_{N, \beta}$ be the probability measure that has probability density proportional to $\exp(-\beta H_N(x))$ with respect to~$\lambda_N$. Let 
\[
Z_N(\beta) := \int_{\mx_N} \exp(-\beta H_N(x))d\lambda_N(x)
\]
be the normalizing constant of $G_{N,\beta}$, and assume that these quantities are finite. Let 
\[
F_N(\beta) := \log Z_N(\beta)\,.
\]
Let $\{L_N\}_{N\ge 1}$ be a sequence of numbers tending to infinity, and let
\[
p(\beta) := \lim_{N\ra\infty} \frac{F_N(\beta)}{L_N}
\]
whenever the limit exists and is finite. For a suitable choice of $L_N$ depending on the situation, the function $p(\beta)$ is sometimes referred to as the thermodynamic limit (or the thermodynamic free energy) of the sequence of systems described above. The thermodynamic limit is said to have a $k^{\mathrm{th}}$ order phase transition at an inverse temperature $\beta$ if the first $k-1$ derivatives of $p$ are continuous at $\beta$ but the $k^{\mathrm{th}}$ derivative is discontinuous at $\beta$.

Fix two inverse temperatures $\beta_0$ and $\beta$ such that $\beta_0< \beta$. The goal is to estimate $F_N(\beta)$ using importance sampling with a sample of size $n$ from the Gibbs measure $G_{N,\beta_0}$, and determine how fast $n$ needs to grow with $N$ such that the ratio of this estimate and the true value tends to one as $N\ra \infty$. Recall that the importance sampling estimate of $Z_N(\beta)$ is
\[
\hat{Z}_{n,N}(\beta) = \frac{Z_N(\beta_0)}{n}\sum_{i=1}^n\exp(-(\beta -\beta_0)H_N(X_i))\,,
\]
where $X_1,\ldots, X_n$ are i.i.d.~draws from $G_{N, \beta_0}$. The following theorem identifies the sample size required for good performance of the above estimate as long as the system does not exhibit a first-order phase transition at $\beta$ in the thermodynamic limit.
\begin{thm}\label{statmechthm}
Let all notation be as above. Let $\{L_N\}_{N\ge 1}$ be a sequence of constants such that the thermodynamic free energy $p$ exists and is differentiable in a neighborhood of $\beta$, and exists at $\beta_0$. Assume that the derivative $p'$ is continuous at $\beta$, and that there exists a finite constant $C$ such that for all $N$ and all $x\in \mx_N$, $|H_N(x)|\le CL_N$. Suppose that the sample size $n = n(N)$ grows with $N$ in such a way that $L_N^{-1}\log n$ converges to a limit $b\in [0,\infty]$, and let 
\[
q(\beta) := p(\beta_0)-p(\beta)-(\beta_0-\beta)p'(\beta)\,.
\]
Then the following conclusions hold:
\begin{enumerate}
\item[\textup{(i)}] If $b> q(\beta)$, then $\hat{Z}_{n,N}(\beta)/Z_N(\beta) \ra 1$ in probability as $N\ra\infty$.
\item[\textup{(ii)}] If $b < q(\beta)$, then $\hat{Z}_{n,N}(\beta)/Z_N(\beta) \not\ra 1$ in probability as $N\ra\infty$.
\item[\textup{(iii)}] If $b = q(\beta)$ and $p'$ is not constant in any neighborhood of $\beta$, then $L_N^{-1}\log \hat{Z}_{n,N}(\beta) \ra p(\beta)$ in probability as $N\ra\infty$. Note that this is a weaker version of the conclusion of part \textup{(i)}.
\end{enumerate}
\end{thm}
Theorem \ref{statmechthm} has potentially much wider  applicability than Theorem \ref{gibbsthm}, since thermodynamic limits are known in many important statistical mechanical systems. Classical examples from statistical physics include the 2D Ising model, the six and eight vertex models, and many others (see \citet{baxter82} and \citet{mccoy10}). Recently, a variety of exponential random graph models have been explicitly `solved' (see \citet{chatterjeediaconis13}, \citet{kenyonetal14}, \citet{kenyonyin14} and \citet{bhattacharyaetal15}). Similar progress has been made for non-uniform distributions on permutations (see \citet{starr09}, \citet{mukherjee13} and \citet{kenyonetal15}). All of these models provide examples for our theory. 


The main strength of Theorem \ref{statmechthm} is also its main weakness: While it gives a definitive answer for exactly solvable models, the theorem is not  useful for systems that are not exactly solvable in a thermodynamic limit. As discussed in Section \ref{problem}, what a practitioner really wants is a diagnostic test that will confirm whether importance sampling has converged. Interestingly, it turns out that the use of the alternative diagnostic test proposed in Section \ref{problem} can be partially justified in the setting of Theorem \ref{statmechthm}, under one additional assumption. The extra assumption is that the system has no first-order phase transition at any point between $\beta_0$ and $\beta$, strengthening the assumption made in Theorem \ref{statmechthm} that there is no first-order phase transition at $\beta$. 

Take $\beta_0$ and $\beta$ such that $\beta_0<\beta$. Recall the quantities $Q_n$ and $q_n$ defined in Section \ref{problem}. Since there are two parameters $n$ and $N$ involved here, we will write $q_{n,N}$ and $Q_{n,N}$ instead of $q_n$ and $Q_n$. Then note that
\[
Q_{n,N} = \frac{\max_{1\le i\le n}\exp(-(\beta-\beta_0)H_N(X_i))}{\sum_{i=1}^n \exp(-(\beta-\beta_0)H_N(X_i))}\,.
\]
and $q_{n,N}=\ee(Q_{n,N})$. (Note that $q_{n,N}$ has nothing to do with $q(\beta)$.) The following theorem shows that if $n$ is large enough (depending on $N$) for the importance sampling to work, then $q_{n,N}$ is exponentially small in $L_N$.  Otherwise, it is not exponentially small.
\begin{thm}\label{newthm}
Let all notation and assumptions be as in Theorem \ref{statmechthm}. Additionally, assume that there is an open interval $I\supseteq[\beta_0,\beta]$ such that the thermodynamic free energy $p$ is well-defined and continuously differentiable in $I$, and that $p'$ is not constant in any nonempty open subinterval of $I$. Then:
\begin{enumerate}
\item[\textup{(i)}] If $b\le q(\beta)$, then 
\[
\lim_{N\ra\infty}\frac{\log q_{n,N}}{L_N} = 0\,.
\]
Moreover, $L_N^{-1}\log Q_{n,N} \ra 0$ in probability as $N\ra\infty$. 
\item[\textup{(ii)}] If $b > q(\beta)$, then
\[
\limsup_{N\ra\infty} \frac{\log q_{n,N}}{L_N} < 0\,.
\]
Moreover, there exists $c < 0$ such that $\pp(L_N^{-1}\log Q_{n,N} \le c) \ra 1$ as $N\ra\infty$.
\end{enumerate}
In particular, if $n$ grows with $N$ so fast that $q_{n,N}$ decays to zero like a negative power of $L_N$, then the estimated free energy $L_N^{-1}\log \hat{Z}_{n,N}(\beta)$ converges to the correct limit $p(\beta)$ in probability. 
\end{thm}
Incidentally, the binomial distribution, as well as more complicated systems like Knuth's self-avoiding paths, can be put into the framework of Theorem \ref{newthm} by an appropriate choice of the Hamiltonian and the inverse temperatures $\beta_0$ and $\beta$, so that the system at inverse temperature $\beta_0$ gives the sampling distribution and the system at inverse temperature $\beta$ gives the target distribution. The main theoretical question would be to prove the absence of a phase transition between $\beta_0$ and $\beta$.

\section{Proofs}\label{proofs}
\begin{proof}[Proof of Theorem \ref{impthm}]
Suppose that $n = e^{L+t}$ and let $a := e^{L+t/2}$. Let $h(x) = f(x)$ if $\rho(x)\le a$ and $0$ otherwise.
Then 
\[
|I_n(f)-I(f)| \le |I_n(f)-I_n(h)| + |I_n(h)-I(h)|+|I(h)-I(f)|\, .
\]
First, note that by the Cauchy--Schwarz inequality,
\begin{align*}
|I(h)-I(f)|&\le \ee(|f(Y)|;\rho(Y)>a)\le \|f\|_{L^2(\nu)}\sqrt{\pp(\rho(Y) >a)}\,.
\end{align*}
Similarly, 
\begin{align*}
\ee|I_n(f)-I_n(h)| &\le \ee|\rho(X_1) f(X_1)-\rho(X_1)h(X_1)|\\
&= \ee(|f(Y)|;\rho(Y)>a)\le \|f\|_{L^2(\nu)}\sqrt{\pp(\rho(Y) >a)}\,.
\end{align*}
Finally, note that 
\begin{align*}
\ee|I_n(h)-I(h)| &\le \sqrt{\var(I_n(h))}\\
&= \sqrt{\frac{\var(\rho(X_1)h(X_1))}{n}}\\
&\le \sqrt{\frac{\ee(\rho(X_1)^2h(X_1)^2)}{n}}\\
&\le \sqrt{\frac{a\, \ee(\rho(X_1)f(X_1)^2)}{n}}\\
&= \|f\|_{L^2(\nu)} \sqrt{\frac{a}{n}}\, .
\end{align*}
Combining the upper bounds obtained above, we get the first inequality in the statement of the theorem. 

Next, suppose that $n= e^{L-t}$ and let $a = e^{L-t/2}$. Markov's inequality gives 
\begin{align}\label{rhoineq}
\pp(\rho(X_1) > a) \le \frac{\ee(\rho(X_1))}{a} = \frac{1}{a}\, .
\end{align} 
Also,
\begin{align*}
\ee(\rho(X_1);\rho(X_1)\le a) = \pp(\rho(Y) \le a)\,.
\end{align*}
Thus,
\begin{align*}
&\pp(I_n(1) \ge 1-\delta) \\
&\le \pp\bigl(\max_{1\le i\le n} \rho(X_i)> a\bigr) + \pp\biggl(\frac{1}{n}\sum_{i=1}^n \rho(X_i)1_{\{\rho(X_i)\le a\}} \ge 1-\delta\biggr)\\
&\le \sum_{i=1}^n \pp(\rho(X_i)> a) + \frac{1}{1-\delta} \ee\biggl(\frac{1}{n}\sum_{i=1}^n \rho(X_i)1_{\{\rho(X_i)\le a\}}\biggr)\\
&\le \frac{n}{a}+\frac{\pp(\rho(Y) \le a)}{1-\delta}\, .
\end{align*}
This completes the proof of the second inequality in the statement of the theorem. 
\end{proof}
\begin{proof}[Proof of Theorem \ref{selfimpthm}]
Suppose that $n = e^{L+t}$ and let $a = e^{L+t/2}$. Let
\[
b := \sqrt{\frac{a}{n}} + 2\,\sqrt{\pp(\rho(Y)>a)}\,.
\]
Then by Theorem \ref{impthm}, for any $\ep,\delta\in(0,1)$,
\[
\pp(|I_n(1)-1|\ge \ep) \le \frac{b}{\ep}
\]
and
\[
\pp(|I_n(f)-I(f)|\ge \delta)\le \frac{\|f\|_{L^2(\nu)}b}{\delta}\,.
\]
Now, if $|I_n(f)-I(f)|<\delta$ and $|I_n(1)-1|< \ep$, then
\begin{align*}
|J_n(f)-I(f)| &= \biggl|\frac{I_n(f)}{I_n(1)} - I(f)\biggr|\\
&\le \frac{|I_n(f)-I(f)| +|I(f)||1-I_n(1)|}{I_n(1)}\\
&< \frac{\delta + |I(f)| \ep}{1-\ep}\,.
\end{align*}
Taking $\ep = \sqrt{b}$ and $\delta = \|f\|_{L^2(\nu)}\ep$ completes the proof of the first inequality in the statement of the theorem. Note that if $\ep$ turns out to be bigger than $1$, then the bound is true anyway. 

Next, suppose that $n = e^{L-t}$ and let $a=e^{L-t/2}$. Let $f(x) = 1$ if $\rho(x)\le a$ and $0$ otherwise. Then $I(f)=\pp(\rho(Y) \le a)$ and  by \eqref{rhoineq}, 
\begin{align*}
\pp(J_n(f)\ne 1) &\le \sum_{i=1}^n \pp(\rho(X_i)> a) \le \frac{n}{a}\,.
\end{align*}
This completes the proof of the theorem. 
\end{proof}
\begin{proof}[Proof of Theorem \ref{rarethm}]
Let
\[
K_n := \frac{I_n(1_A)}{\nu(A)} = \frac{1}{n}\sum_{i=1}^n \rho_A(X_i)\, .
\]
Suppose that $n= e^{L+t}$ and let $a = e^{L+t/2}$. Applying Theorem \ref{impthm} with $\rho$ replaced by $\rho_A$, this gives
\begin{align*}
\ee|K_n - 1| \le \sqrt{\frac{a}{n}} + 2\sqrt{\pp(\rho_A(Y)>a\mid Y\in A)} \,,
\end{align*}
which is the first assertion of the theorem. The second claim follows similarly.
\end{proof}

\begin{proof}[Proof of Proposition \ref{rareprop}]
Let 
\[
Z= \frac{1}{2^N}\sum_{Np\le j\le N}{N\choose j}\,,
\]
so that 
\[
\nu_A(j) = \frac{{N\choose j}}{2^N Z} 1_A(j)\,.
\]
To explore the choice of sampling distribution let $\mu$ be the Binomial$(N, \theta)$ distribution for fixed $1/2<\theta< 1$. Then 
\begin{align*}
L_A &= D(\nu_A\|\mu)=\sum_j \log(\nu_A(j)/\mu(j)) \nu_A(j)\\
&= -\frac{1}{2^NZ} \sum_{Np\le j\le N} \log(2^N Z \theta^j (1-\theta)^{N-j}) {N\choose j}\\
&= -\log(2^N Z (1-\theta)^N) - \frac{\log(\theta/(1-\theta))}{2^N Z} \sum_{Np\le j\le N} j {N \choose j}\,.
\end{align*}
An identity of de Moivre (see \citet{diaconiszabell}) shows that for any $k$, $0\le k\le N$, 
\[
\frac{1}{2^N} \sum_{k\le j\le N} {N\choose j}\biggl(j-\frac{N}{2}\biggr) = \frac{k}{2} b(k; N, 1/2)\,.
\]
Thus, since $Np$ is an integer, 
\[
L_A = -\log(2^N Z(1-\theta)^N) - \log(\theta/(1-\theta))\biggl(\frac{Np}{2Z} b(Np; N, 1/2) + \frac{N}{2}\biggr)\,.
\]
To approximate $Z$, use an inequality of \citet{bahadur60}, specialized here: Let 
\[
R = \frac{1}{2}b(Np; N, 1/2)\frac{Np+1}{Np+1-(N+1)/2}\,.
\]
Then 
\[
1\le \frac{R}{Z}\le 1+x^{-2}\,,
\]
where 
\[
x = \frac{Np-N/2}{\sqrt{N/4}}\,.
\]
For large $N$ and $p$ fixed, this gives
\[
Z\sim b(Np; N, 1/2)\frac{p}{2p-1}\,.
\]
Stirling's formula gives
\[
2^Nb(Np; N,1/2) \sim \frac{(p^p(1-p)^{1-p})^N}{\sqrt{2\pi N p(1-p)}}\,.
\]
Putting these approximations into $L_A$, we get
\[
L_A \sim -N \log(p^p (1-p)^{1-p} (1-\theta) (\theta/(1-\theta))^p)\,.
\]
The right side, as a function of $\theta$, is minimized when $\theta=p$. Plugging this in gives the claim. 
\end{proof}

\begin{proof}[Proof of Theorem \ref{flawthm}]
Let $X$ be a random variable with law $\mu$. Then note that 
\begin{align*}
1 &= \ee(\rho(X)) = \int_0^\infty \pp(\rho(X) \ge t) dt\\
&\ge \sum_{k=0}^\infty \int_{2^k}^{2^{k+1}}\pp(\rho(X)\ge t) dt\,.
\end{align*}
Therefore, for any $l\ge 0$,
\[
\min_{0\le k\le l} \int_{2^k}^{2^{k+1}}\pp(\rho(X)\ge t) dt\le \frac{1}{l+1}\sum_{k=0}^l\int_{2^k}^{2^{k+1}}\pp(\rho(X)\ge t) dt\le \frac{1}{l+1}\,.
\]
Thus, there exists $k\le l$ such that 
\[
\int_{2^k}^{2^{k+1}}\pp(\rho(X)\ge t) dt\le \frac{1}{l+1}\,.
\]
Fixing $l$, take any such $k$. The above inequality implies that there exists $t\in [2^k, 2^{k+1}]$ such that 
\[
\pp(\rho(X)\ge t) \le \frac{1}{(l+1)2^k}\,.
\]
Now take any $\ep > 0$. Let  $l = [1/\ep^3]$, where $[1/\ep^3]$ is the integer part of $1/\ep^3$. Then there exists $k\le l$ and $t\in [2^k, 2^{k+1}]$ such that the above inequality is satisfied. Let $n = [t/\ep^2]+1$. Then 
\begin{align*}
\pp(\max_{1\le i\le n} \rho(X_i)\ge \ep^2 n) &\le n \pp(\rho(X)\ge \ep^2 n) \\
&\le n \pp(\rho(X)\ge t)\\
&\le \frac{n}{(l+1)2^k}\le \frac{\ep^3 n }{2^k}\le \frac{\ep (t+1)}{2^k} \le \frac{\ep(2^{k+1}+1)}{2^k}\le 3\ep\,.
\end{align*}
Consequently, for this $n$, 
\begin{align*}
\pp(v_n(f)\ge \ep) &\le \pp\biggl(\frac{1}{n^2} \sum_{i=1}^n f(X_i)^2\rho(X_i)^2\ge \ep\biggr)\\
&\le \pp(\max_{1\le i\le n} \rho(X_i)\ge \ep^2n) + \pp\biggl(\frac{1}{n} \sum_{i=1}^n f(X_i)^2\rho(X_i)\ge \frac{1}{\ep}\biggr)\\
&\le 3\ep + \ep\, \ee\biggl(\frac{1}{n} \sum_{i=1}^n f(X_i)^2\rho(X_i)\biggr) \\
&= 3\ep + \ep\, \ee(f(X)^2\rho(X)) = 3\ep + \ep \|f\|_{L^2(\nu)}^2\,.
\end{align*}
To complete the proof, note that $n\le \ep^{-2}t \le \ep^{-2}2^{k+1} \le \ep^{-2}2^{l+1}\le \ep^{-2}2^{1+\ep^{-3}}$.
\end{proof}

\begin{proof}[Proof of Theorem \ref{qthm}]
Since $0\le Q_n\le 1$, therefore
\begin{align*}
q_n = \ee(Q_n) &\le \frac{q_n}{2} +\pp\biggl(Q_n \ge \frac{q_n}{2}\biggr)\\
&= \frac{q_n}{2} + \pp\biggl(M_n \ge \frac{q_n S_n}{2}\biggr)\\
&\le \frac{q_n}{2} + \pp\biggl(S_n < \frac{n}{2}\biggr)+\pp\biggl(M_n \ge \frac{q_n n}{4}\biggr)\\
&\le \frac{q_n}{2} + 2\ep_n + \pp\biggl(M_n \ge \frac{q_n n}{4}\biggr)\,.
\end{align*}
Suppose that 
\begin{align}\label{qn8}
\ep_n\le \frac{q_n}{8}\,.
\end{align}
Then by the previous display,
\begin{align}\label{qn4}
\pp\biggl(M_n \ge \frac{q_n n}{4}\biggr) \ge \frac{q_n}{4}\,.
\end{align}
Let $k := [8/q_n]$ and $l:= [n/k]$. For $1\le j\le k$, define
\[
M_{n,j} := \max_{(j-1)l+1\le i\le jl} \rho(X_i)\,.
\] 
Then for any $x\ge 0$,
\begin{align*}
\pp(S_n \ge kx) &\ge \pp(M_{n,j} \ge x \text{ for all } 1\le j\le k)\\
&= (\pp(M_{n,1} \ge x))^k\\
&= (1-(\pp(\rho(X_1)< x))^l)^k\\
&= (1-(\pp(M_n< x))^{l/n})^k\\
&= (1-(1-\pp(M_n\ge x))^{l/n})^k\,.
\end{align*}
Since $k\le 8/q_n$ and $l\ge n/k - 1\ge q_n n/8 -1$, this gives
\[
\pp(S_n \ge kx)\ge (1-(1-\pp(M_n\ge x))^{q_n/8 - 1/n})^{8/q_n}\,.
\]
Suppose that 
\begin{equation}\label{qn9}
\frac{1}{n}\le \frac{q_n}{16}\,.
\end{equation}
Then the previous equation gives
\[
\pp(S_n \ge kx)\ge (1-(1-\pp(M_n\ge x))^{q_n/16})^{8/q_n}\,.
\]
Taking $x = q_nn/4$, assuming \eqref{qn8} and \eqref{qn9}, and using \eqref{qn4}, gives 
\begin{align*}
\pp(S_n \ge 2n) \ge (1-(1-q_n/4)^{q_n/16})^{8/q_n}\,.
\end{align*}
Now note that $1-(1-y)^{y/4}$ is asymptotic to $y^2/4$ as $y\ra 0$, and is positive everywhere in the interval $(0,1)$. Therefore there is a positive constant $C_1$ such that $1-(1-y)^{y/4}\ge C_1y^2$ for all $y\in [0,1]$. Using this in the above inequality gives 
\[
\pp(S_n \ge 2n) \ge e^{-8q_n^{-1}\log(C_2/q_n)}\,,
\] 
where $C_2$ is a universal constant. By Markov's inequality, $\pp(S_n \ge 2n)\le \ep_n$. Therefore
\[
e^{-8q_n^{-1}\log(C_2/q_n)}\le \ep_n\,.
\]
This shows that as $\ep_n\ra0$, $q_n$ must also tend to zero. Using this and the monotonicity of the map $x\mapsto (\log x)/x$ for $x\ge e$, it is easy to show that 
\[
q_n \le \frac{C_3 \log \log (1/\ep_n)}{\log (1/\ep_n)}\,,
\]
where $C_3$ is a universal constant. 
Note that this holds under \eqref{qn8} and \eqref{qn9}. The maximum in the statement of the theorem accounts for these constraints. 
\end{proof}

\begin{proof}[Proof of Theorem \ref{gibbsthm}]
By the integrability condition on $H$ and the dominated convergence theorem, it is easy to see that $F$ is infinitely differentiable. Moreover, if $Y$ is a random variable with law $G_\beta$, then 
\begin{equation}\label{fp}
F'(\beta) = -\ee(H(Y))
\end{equation}
and 
\begin{equation}\label{fpp}
F''(\beta) = \var(H(Y))\,.
\end{equation}
The probability density of $G_\beta$ with respect to $G_{\beta_0}$ is 
\[
\rho(x) = \frac{Z(\beta_0)}{Z(\beta)} \exp(-(\beta-\beta_0) H(x))\,.
\]
Therefore
\[
\frac{\hat{Z}_n(\beta)}{Z(\beta)} = \frac{1}{n}\sum_{i=1}^n \rho(X_i)\,.
\]
In the notation of Theorem \ref{impthm}, this is nothing but $I_n(1)$. Now note that if $Y\sim G_\beta$, then by \eqref{fp} and \eqref{fpp},
\begin{align*}
\ee(\log\rho(Y)) &= F(\beta_0)-F(\beta) - (\beta-\beta_0)\ee(H(Y))\\
&= F(\beta_0)-F(\beta) - (\beta_0-\beta) F'(\beta) = L\,,
\end{align*}
and
\begin{align*}
\var(\log \rho(Y)) = (\beta_0-\beta)^2\var(H(Y)) = (\beta_0-\beta)^2 F''(\beta)= \sigma^2\,.
\end{align*}
The proof is now easily completed by an application of Theorem \ref{impthm}, together with Chebychev's inequality for bounding the tail probabilities occurring in the statement of Theorem \ref{impthm}.
\end{proof}

\begin{proof}[Proof of Theorem \ref{statmechthm}]
Let $\rho_N$ be the probability density of $G_{N,\beta}$ with respect to $G_{N, \beta_0}$. As in the proof of Theorem \ref{gibbsthm}, we have 
\begin{equation}\label{later1}
\log \rho_N(x) = F_N(\beta_0)-F_N(\beta) -(\beta-\beta_0)H_N(x)\,.
\end{equation}
For each $\gamma$, let $Y_{N,\gamma}$ be a random variable with law $G_{N,\gamma}$. A simple computation shows that for any  bounded measurable function $\phi:\rr\ra\rr$,
\[
\frac{d}{d\gamma} \ee(\phi(H_N(Y_{N,\gamma}))) = \cov(\phi(H_N(Y_{N,\gamma})), H_N(Y_{N,\gamma}))\,.
\] 
It is an easy fact that if $X$ is a real-valued random variable and $f$ and $g$ are two increasing functions, then $\cov(f(X), g(X))\ge 0$. From this and the above identity, it follows that for any bounded increasing function $\phi$, 
\[
\frac{d}{d\gamma} \ee(\phi(H_N(Y_{N,\gamma})))\ge 0\,.
\]
In particular, for any $t\in \rr$, $\pp(H_N(Y_{N,\gamma})\ge t)$ is an increasing function of $\gamma$. This is an important observation that will be used below.

Take any $\gamma$ such that $p$ is well-defined and differentiable in an open  neighborhood of $\gamma$. Note that $F_N$ is a convex function, since $F_N''$ is nonnegative by \eqref{fpp}. Therefore for any $h>0$,
\[
F_N'(\gamma) \le \frac{F_N(\gamma+h)-F_N(\gamma)}{h}\,.
\]
Consequently, if $h$ is small enough, then 
\[
\limsup_{N\ra\infty} \frac{F_N'(\gamma)}{L_N} \le \frac{p(\gamma+h)-p(\gamma)}{h}\,.
\]
Taking $h\ra 0$, we get
\[
\limsup_{N\ra\infty} \frac{F_N'(\gamma)}{L_N} \le p'(\gamma)\,.
\]
Similarly,
\[
\liminf_{N\ra\infty} \frac{F_N'(\gamma)}{L_N} \ge p'(\gamma)\,.
\]
This proves that for all $\gamma$ in an open neighborhood of $\beta$,
\[
\lim_{N\ra\infty} \frac{F_N'(\gamma)}{L_N} = p'(\gamma)\,.
\]
Using the monotonicity of $F_N'$ and $p'$ and the continuity of $p'$ at $\beta$, it is easy to conclude from the above identity that for any sequence $\gamma_N\ra\beta$,
\begin{equation}\label{fpn}
\lim_{N\ra\infty} \frac{F_N'(\gamma_N)}{L_N} = p'(\beta)\,.
\end{equation}
By \eqref{fp}, note that for any $\gamma$ 
\[
|F_N'(\gamma)| = |\ee(H_N(Y_{N,\gamma}))|\le CL_N\,.
\]
Therefore
\[
\int_\beta^{\beta+L_N^{-1/2}}F_N''(\gamma) d\gamma = F_N'(\beta+L_N^{-1/2})-F_N'(\beta)\le 2CL_N\,.
\]
Thus, there exists $\gamma_N\in [\beta, \beta+L_N^{-1/2}]$ such that 
\begin{equation}\label{fppn}
F_N''(\gamma_N)\le 2CL_N^{3/2}\,.
\end{equation}
Since $L_N\ra\infty$, therefore $\gamma_N \ra\beta$. Hence by \eqref{fp}, \eqref{fpp}, \eqref{fpn} and \eqref{fppn},
\begin{equation}\label{later2}
\lim_{N\ra\infty} \ee\biggl(\frac{-H_N(Y_{N,\gamma_N})}{L_N}\biggr) = \lim_{N\ra\infty}\frac{F_N'(\gamma_N)}{L_N} = p'(\beta)
\end{equation}
and
\[
\lim_{N\ra\infty} \var\biggl(\frac{-H_N(Y_{N,\gamma_N})}{L_N}\biggr) = \lim_{N\ra\infty}\frac{F_N''(\gamma_N)}{L_N^2} =0\,.
\]
This implies that $-H_N(Y_{N, \gamma_N})/L_N \ra p'(\beta)$ in probability. Therefore, by our previous observation about the monotonicity of tail probabilities, 
\[
\lim_{N\ra\infty} \pp\biggl(\frac{-H_N(Y_{N,\beta})}{L_N} \ge p'(\beta)+\delta\biggr) = 0
\]
for any $\delta>0$. In a similar manner, one can show that
\[
\lim_{N\ra\infty} \pp\biggl(\frac{-H_N(Y_{N,\beta})}{L_N} \le p'(\beta)-\delta\biggr) = 0\,.
\]
Thus, $-H_N(Y_{N,\beta})/L_N \ra p'(\beta)$ in probability. Consequently, 
\[
\frac{\log \rho_N(Y_{N,\beta})}{L_N} \ra p(\beta_0)-p(\beta)-(\beta_0-\beta)p'(\beta) = q(\beta)
\]
in probability. The proofs of parts (i) and (ii) are now easily completed by applying Theorem~\ref{impthm}. To prove part (iii), take any $\gamma\in (\beta_0,\beta)$. Since $p'$ is nonconstant in any neighborhood of $\beta$ and $p'$ is an increasing function due to the convexity of $p$, therefore $p'(\gamma)< p'(\beta)$. Thus, by the convexity of $p$,
\begin{align}
q(\beta)-q(\gamma) &=(\beta-\beta_0) (p'(\beta)-p'(\gamma)) + (\beta-\gamma)p'(\gamma) + p(\gamma)-p(\beta) \nonumber\\
&\ge (\beta-\beta_0) (p'(\beta)-p'(\gamma)) + (\beta-\gamma) (p'(\gamma) - p'(\beta))\nonumber\\
&= (\gamma-\beta_0)(p'(\beta)-p'(\gamma)) > 0\,.\label{qinc}
\end{align}
By part (i) of the theorem, this implies that if $b = q(\beta)$, then 
\[
\frac{\hat{Z}_{n,N}(\gamma)}{Z_N(\gamma)} \ra 1
\]
in probability, and therefore 
\begin{equation}\label{probconv}
\frac{\log \hat{Z}_{n,N}(\gamma)}{L_N} \ra p(\gamma)
\end{equation}
in probability. Now note that for any $\beta'$, 
\[
\biggl|\frac{d}{d\beta'} \log \hat{Z}_{n,N}(\beta')\biggr| = \biggl|\frac{\sum_{i=1}^n H_N(X_i)\exp(-(\beta'-\beta_0)H_N(X_i))}{\sum_{i=1}^n \exp(-(\beta'-\beta_0)H_N(X_i))}\biggr| \le CL_N\,.
\]
Therefore
\begin{equation}\label{zbd}
|\log \hat{Z}_{n,N}(\beta) - \log \hat{Z}_{n,N}(\gamma)|\le CL_N (\beta-\gamma)\,.
\end{equation}
Since $\gamma$ is an arbitrary point in $(\beta_0,\beta)$, it is now easy to complete the proof of part (iii) using \eqref{probconv}, \eqref{zbd} and the continuity of $p$.
\end{proof}
\begin{proof}[Proof of Theorem \ref{newthm}]
Suppose that $\{W_N\}_{N\ge 1}$ is a sequence of real-valued random variables and $c$ is a real number. In this proof, we will use the notation
\[
\textup{P-}\liminf_{N\ra\infty} W_N \ge c
\]
to mean that for any $\ep >0$, $\lim_{N\ra\infty}\pp(W_N\ge c-\ep) = 1$. Similarly, $\textup{P-}\limsup_{N\ra\infty} W_N \le c$ means that for any $\ep>0$, $\lim_{N\ra\infty}\pp(W_N\le c+\ep) = 1$, and $\textup{P-}\lim_{N\ra\infty} W_N = c$ means that both of these hold, that is, $W_N\ra c$ in probability. 

First, suppose that $b\le q(\beta)$. Since $p$ has no interval of linear behavior in the interval $I$, therefore the convexity of $p$ implies that $p'$ is strictly increasing in $I$. From this and a variant of \eqref{qinc} it is easy to see that  in the interval $I\cap[\beta_0,\infty)$, $q$ is continuous and strictly increasing. Moreover, $q(\beta_0)=0$. It follows that for any $a\in [0,q(\beta)]$, there exists $\gamma\in[\beta_0,\beta]$ such that $q(\gamma)=a$. Therefore, since $b \le q(\beta)$, therefore $b = q(\gamma)$ for some $\gamma\in [\beta_0,\beta]$. Suppose that $\gamma > \beta_0$. Then by part (i) and part (iii) of Theorem \ref{statmechthm}, 
\begin{align}
&\textup{P-}\lim_{N\ra\infty}\frac{1}{L_N} \log\biggl(\sum_{i=1}^n \exp(-(\gamma -\beta_0)H_N(X_i))\biggr) \nonumber \\
&= q(\gamma)+p(\gamma)-p(\beta_0) = (\gamma-\beta_0)p'(\gamma)\,. \label{znb}
\end{align}
Let $U_N(\gamma)$ denote the left-hand side of \eqref{znb}, without the limit. Using the positivity of the second derivative, it is easy to see that $U_N$ is a convex function of $\gamma$. Take any $\gamma'\in (\beta_0,\gamma)$. Then by the convexity of $U_N$, we have
\begin{align*}
\frac{\max_{1\le i\le n} (-H_N(X_i))}{L_N} &\ge \frac{1}{L_N}\frac{\sum_{i=1}^n (-H_N(X_i)) \exp(-(\gamma-\beta_0) H_N(X_i))}{\sum_{i=1}^n \exp(-(\gamma-\beta_0) H_N(X_i))}\\
&= U_N'(\gamma) \\
&\ge \frac{U_N(\gamma)-U_N(\gamma')}{\gamma-\gamma'}\,.
\end{align*}
Now let $N\ra\infty$ on both sides and apply \eqref{znb}, and then let $\gamma'\ra\gamma$ on the right. This gives
\begin{align}\label{plim1}
\textup{P-}\liminf_{N\ra\infty}\frac{\max_{1\le i\le n} (-H_N(X_i))}{L_N} &\ge p'(\gamma)\,.
\end{align}
Next, note that  
\begin{align*}
&\log\biggl(\sum_{i=1}^n \exp(-(\beta-\beta_0)H_N(X_i))\biggr) \\
&\le (\beta-\gamma) \max_{1\le i\le n} (-H_N(X_i)) + \log\biggl(\sum_{i=1}^n \exp(-(\gamma-\beta_0)H(X_i))\biggr)
\end{align*}
By \eqref{znb} and \eqref{plim1}, this implies that
\begin{equation}\label{q1}
\textup{P-}\limsup_{N\ra\infty}\frac{- \log Q_{n,N}}{L_N} \le 0\,.
\end{equation}
Note that this inequality was proved under the assumption that $\gamma > \beta_0$. Next, suppose that $\gamma=\beta_0$. Observe the easy inequality
\begin{align*}
\log\biggl(\sum_{i=1}^n \exp(-(\beta-\beta_0)H_N(X_i))\biggr) &\le (\beta-\beta_0) \max_{1\le i\le n} (-H_N(X_i)) + \log n\,.
\end{align*}
From this and the fact that $L_N^{-1}\log n \ra q(\beta_0)=0$, it follows that \eqref{q1} holds even if $\gamma=\beta_0$.  Next, note that we trivially have
\begin{align*}
\log\biggl(\sum_{i=1}^n \exp(-(\beta-\beta_0)H_N(X_i))\biggr) &\ge \log\biggl(\max_{1\le i\le n} \exp(-(\beta-\beta_0)H_N(X_i))\biggr)\,
\end{align*}
which is same as 
\begin{equation}\label{q2}
\textup{P-}\liminf_{N\ra\infty}\frac{- \log Q_{n,N}(\beta)}{L_N} \ge 0\,.
\end{equation}
Equations \eqref{q1} and \eqref{q2} prove that if $b\le q(\beta)$, then $L_N^{-1}\log Q_{n,N}\ra 0$ in probability. Next, note that $Q_{n,N}\in [0,1]$, which implies that $\ee(Q_{n,N})\in [0,1]$ and hence 
\begin{equation}\label{e1}
\frac{\log \ee(Q_{n,N})}{L_N}\le 0\,.
\end{equation} 
On the other hand, Jensen's inequality gives 
\begin{equation}\label{e2}
\frac{\log \ee(Q_{n,N})}{L_N} \ge \frac{\ee(\log Q_{n,N})}{L_N}\,.
\end{equation}
It is not difficult to see that since $|H_N|\le CL_N$ and $L_N^{-1}\log n \ra b<\infty$, therefore the random variable $|L_N^{-1}\log Q_{n,N}|$ is bounded by a non-random constant that does not vary with $N$. Since we already know that 
\[
L_N^{-1}\log Q_{n,N} \ra 0
\]
in  probability, this shows that
\[
\lim_{N\ra\infty}\frac{\ee(\log Q_{n,N})}{L_N} = 0\,.
\]
Combining this with \eqref{e1} and \eqref{e2}, we get 
\[
\lim_{N\ra\infty} \frac{\log \ee(Q_{n,N})}{L_N} = 0\,.
\]
This completes the proof of part (i) of the theorem. Next, suppose that $b> q(\beta)$. Then note that by Theorem~\ref{statmechthm},
\begin{align}\label{second1}
\textup{P-}\lim_{N\ra\infty} \frac{Z_N(\beta_0)}{n Z_N(\beta)} \sum_{i=1}^n \exp(-(\beta-\beta_0)H_N(X_i)) = 1\,.
\end{align}
Next, let
\[
M_N := \max_{1\le i\le n} \exp(-(\beta-\beta_0)H_N(X_i))\,.
\]
Since $p$ is continuously differentiable in the interval $I$ and $p'$ is strictly increasing, therefore there exists $\gamma \in I\cap (\beta, \infty)$ such that $b > q(\gamma)$. 
If $M_N > \exp(L_N (\beta-\beta_0)p'(\gamma))$, then 
\begin{align*}
M_N \le \sum_{i=1}^n \exp(-(\beta-\beta_0)H_N(X_i))1_{\{-H_N(X_i)> L_Np'(\gamma)\}} =:M_N'\,.
\end{align*}
Therefore
\begin{align}\label{mn1}
 M_N &\le \max\{\exp(L_N(\beta-\beta_0) p'(\gamma)),\, M_N'\}\,.
\end{align}
Define
\begin{equation}\label{rndef}
R_N := \frac{Z_N(\beta_0)M_N'}{nZ_N(\beta)}\,.
\end{equation}
Note that if $Y$ is a random variable with law $G_{N,\beta}$, then for any $\theta>0$, 
\begin{align}
\ee(R_N) &= \ee\biggl(\frac{Z_N(\beta_0)}{Z_N(\beta)} \exp(-(\beta-\beta_0)H_N(X_1))1_{\{-H_N(X_1)> L_Np'(\gamma)\}}\biggr)\nonumber\\
&= \pp(-H_N(Y)> L_Np'(\gamma))\nonumber\\
&\le e^{-\theta L_N p'(\gamma)} \ee(e^{-\theta H_N(Y)})\nonumber\\
&= \exp(-\theta L_N p'(\gamma) + F_N(\beta+\theta)-F_N(\beta))\,. \label{zn1}
\end{align}
Let
\begin{equation*}\label{cthetadef}
c(\theta) := p(\beta+\theta)-p(\beta) - \theta p'(\gamma)\,,
\end{equation*}
and choose $\theta = (\gamma-\beta)/2$. Then by the strict convexity of $p$ in $I$, 
\begin{align}\label{ctheta}
c(\theta) \le \theta (p'(\beta+\theta) -p'(\gamma))  < 0\,.
\end{align}
By \eqref{zn1} and Markov's inequality,
\begin{align*}
&\pp\biggl(\frac{\log R_N}{L_N}\ge \frac{c(\theta)}{2}\biggr) \le e^{-L_Nc(\theta)/2}\ee(R_N)\\
&\le  \exp\biggl(-\frac{L_Nc(\theta)}{2}  -\theta L_N p'(\gamma) + F_N(\beta+\theta)-F_N(\beta)\biggr)\,.
\end{align*}
Taking logarithm on both sides, dividing by $L_N$ and sending $N\ra\infty$, we get
\begin{equation}\label{later3}
\limsup_{N\ra\infty} \frac{1}{L_N}\log \pp\biggl(\frac{\log R_N}{L_N}\ge \frac{c(\theta)}{2}\biggr)<0\,.
\end{equation}
In particular, 
\begin{align}\label{mn2}
\textup{P-}\limsup_{N\ra\infty}\frac{\log R_N}{L_N}\le \frac{c(\theta)}{2}\,.
\end{align}
Next, note that
\begin{align}
&\lim_{N\ra\infty} \frac{1}{L_N} \log \biggl(\frac{Z_N(\beta_0)}{nZ_N(\beta)}\exp(L_N(\beta-\beta_0)p'(\gamma) )\biggr)\nonumber\\
&= p(\beta_0) -p(\beta)- b + (\beta-\beta_0)p'(\gamma)\nonumber\\
&\le q(\gamma)-b\,.\label{mn3}
\end{align}
From \eqref{mn1}, \eqref{mn2} and \eqref{mn3} we get
\begin{align}\label{second2}
\textup{P-}\limsup_{N\ra\infty} \frac{1}{L_N}\log \biggl(\frac{Z_N(\beta_0)M_N}{n Z_N(\beta)}\biggr) \le \max\biggl\{q(\gamma)-b,\, \frac{c(\theta)}{2}\biggr\}\,.
\end{align}
By combining \eqref{second1}, \eqref{second2}, \eqref{ctheta} and the fact that $q(\gamma)<b$, this shows that there exists $c< 0$ such that $\pp(L_N^{-1}\log Q_{n,N}\le c) \ra 1$ as $N\ra\infty$.

Next, let
\[
V_N := \frac{Z_N(\beta_0)}{n Z_N(\beta)} \sum_{i=1}^n \exp(-(\beta-\beta_0)H_N(X_i))\,.
\]
Then $V_N$ is nothing but the importance sampling estimate $I_n(1)$ when the sampling measure is $G_{N,\beta_0}$ and the target measure is $G_{N,\beta}$.  In this setting, we have already seen in the proof of Theorem~\ref{statmechthm} that the quantity $L$ of Theorem~\ref{impthm} is asymptotic to $L_Nq(\beta)$ (to see this, simply combine the equations \eqref{later1} and \eqref{later2}). Combined with the fact that $L_N^{-1}\log n \ra b$, this implies that the quantity $t$ of Theorem~\ref{impthm} is asymptotic to $L_N(b- q(\beta))$ in the present setting. 

Next, let $Y\sim G_{N,\beta}$ and $\rho_N$ be the probability density of $G_{N,\beta}$ with respect to $G_{N,\beta_0}$. The formula~\eqref{later1} implies that $\log \rho_N(Y)$ is asymptotic to $L_N(p(\beta_0)-p(\beta))- (\beta-\beta_0)H_N(Y)$. Combining all of these observations and applying Theorem \ref{impthm}, it follows that there is a positive constant $c$ (which may depend on $\beta$, $\beta_0$ and $b$) such that for all large enough $N$,
\begin{align*}
\ee|V_N-1| \le e^{-cL_N} + \sqrt{\pp(-H_N(Y) \ge L_N(p'(\beta)+c))}\,.
\end{align*}
Take any $\theta > 0$. Then 
\[
\pp(-H_N(Y) \ge L_N(p'(\beta)+c)) \le e^{-\theta L_N(p'(\beta)+c)} \ee(e^{-\theta H_N(Y)})\,.
\]
It is easy to see that $\log \ee(e^{-\theta H_N(Y)})$ is asymptotic to $L_N(p(\beta+\theta)-p(\beta))$. Thus, the logarithm of the right-hand side in the above display is asymptotic to $-\theta c L_N + L_N(p(\beta+\theta)-p(\beta)-\theta p'(\beta))$. Since $p'$ is continuous in a neighborhood of $\beta$, we can choose a $\theta$ small enough so that $\theta c > p(\beta+\theta)-p(\beta)-\theta p'(\beta)$. Therefore, there exists $C_1 > 0$ such that $\pp(-H_N(Y) \ge L_N(p'(\beta)+c))\le e^{-C_1L_N}$ for all large enough $N$. Combining these steps we see that there is a positive constant $C_2$ such that $\ee|V_N-1|\le e^{-C_2L_N}$ for all large $N$, and hence 
\begin{equation}\label{vn}
\pp(V_N < 1/2)\le 2e^{-C_2L_N}\,. 
\end{equation}
Now note that by \eqref{mn1},
\begin{equation}\label{qnn}
Q_{n,N} \le \frac{\max\{S_N, R_N\}}{V_N}\,,
\end{equation}
where $R_N$ is defined in \eqref{rndef} and 
\[
S_N := \frac{Z_N(\beta_0)}{nZ_N(\beta)}\exp(L_N(\beta-\beta_0) p'(\gamma))\,.
\]
Recall that by \eqref{later3}, there are positive constants $C_3$ and $C_4$ such that for all large enough $N$,
\begin{equation}\label{rn}
\pp(R_N \ge e^{-C_3L_N})\le e^{-C_4L_N}\,.
\end{equation}
Since $Q_{n,N} \in [0,1]$, \eqref{vn}, \eqref{qnn} and \eqref{rn} imply that
\begin{align*}
\ee(Q_{n,N}) &\le \pp(V_N < 1/2) + \pp(R_N \ge e^{-C_3L_N}) + 2\max\{S_N, e^{-C_3L_N}\}\\
&\le 2e^{-C_2L_N} + e^{-C_4L_N} + 2\max\{S_N, e^{-C_3L_N}\}\,.
\end{align*}
However, we have already seen in \eqref{mn3} that there is a constant $C_5>0$ such that $S_N\le e^{-C_5 L_N}$ for all large enough $N$. Thus, 
\[
\limsup_{N\ra \infty} L_N^{-1}\log \ee(Q_{n,N}) < 0.
\]
This completes the proof of the theorem. 
\end{proof}
\vskip.2in
\noindent {\bf Acknowledgments.} We thank Ben Bond, Brad Efron, Jonathan Huggins, Don Knuth, Shuangning Li, Art Owen, Daniel Roy and David Siegmund for helpful comments. We also thank the referees and the associate editor for many helpful suggestions, and the editorial board of AAP for their patience with our long delay in submitting the revision.

\bibliographystyle{plainnat}

\end{document}